\documentclass[a4paper]{amsart}
\usepackage{amscd,amsmath,amssymb,amsfonts,verbatim}%,MnSymbol}
\usepackage[cmtip, all]{xy}
\usepackage{pgf,tikz}
\usepackage{mathrsfs}
\usetikzlibrary{arrows}
\newtheorem{thm}{Theorem}[subsection]
\newtheorem{prop}[thm]{Proposition}
\newtheorem{lem}[thm]{Lemma}
\newtheorem{cor}[thm]{Corollary}

\theoremstyle{definition}

\newtheorem{defn}[thm]{Definition}
\theoremstyle{remark}
\newtheorem{remk}[thm]{Remark}
\newtheorem{remks}[thm]{Remarks}

\newtheorem{exm}[thm]{Example}
\newtheorem{exms}[thm]{Examples}
\newtheorem{notat}[thm]{Notation}
\numberwithin{equation}{subsection}

{\hfill$\square$\end{defn}}
{\hfill$\square$\end{remk}}
{\hfill$\square$\end{remks}}
{\hfill$\square$\end{exm}}
{\hfill$\square$\end{exms}}
{\hfill$\square$\end{notat}}

\newcommand{\CH}{{\rm CH}}

\newcommand{\BGHz}{\mathbf{z}}

\newcommand{\Hom}{{\rm Hom}}

\newcommand{\Spec}{{\rm Spec \,}}

\newcommand{\Sch}{{\operatorname{\mathbf{Sch}}}}

\newcommand{\ds}{{/\kern-3pt/}}
\newcommand{\ess}{{\rm ess}}

\newcommand{\sgn}{{\operatorname{\rm sgn}}}

\newcommand{\un}{\underline}
\newcommand{\ov}{\overline}

\renewcommand{\dim}{\text{\rm dim}}
\newcommand{\tuborg}{\left\{\begin{array}{ll}}
\newcommand{\sluttuborg}{\end{array}\right.}

\begin{document}
\title{On localization for cubical higher Chow groups}
\author{Jinhyun Park}
\address{Department of Mathematical Sciences, KAIST, 291 Daehak-ro Yuseong-gu, Daejeon, 34141, Republic of Korea (South)}
\email{jinhyun@mathsci.kaist.ac.kr; jinhyun@kaist.edu}

%\maketitle



\begin{abstract}
We give a purely cubical argument for the localization theorem for the cubical version of higher Chow groups.
\end{abstract}

\maketitle

\setcounter{tocdepth}{1}

%\tableofcontents

\section{Introduction}
In topology, the singular homology theory of a topological space is a foundational tool. Its typical construction uses the topological simplices $\Delta_{\rm top} ^n:= \{ (t_0, \ldots, t_n) \in \mathbb{R}^{n+1} \ | \ \sum_{i=0} ^n t_i = 1, \ t_i \geq 0\}$ as part of building materials, but there is also a variant based on the cubes $I^n:= [0,1]^n$; see W. Massey \cite{Massey}. This cubical version of the singular homology theory is equivalent to the simplicial version in that their homology groups are isomorphic, while there are also some technical advantages such as the product structure.

In algebraic geometry, for the category $\Sch_k$ of $k$-schemes of finite type over a field $k$, the motivic homology theory can be described by the higher Chow theory of S. Bloch \cite{Bloch HC}, which is based on the algebraic simplices $\Delta^n:= \Spec \left( \frac{ k[t_0, \ldots, t_n]}{\sum_{i=0} ^n t_i -1}\right)$ (see B. Totaro \cite[\S 2, p.84]{Totaro2} and other references there for a relevant detailed discussion). It does have its cubical version based on the algebraic cubes $\square^n_k$, where $\square_k:= \mathbb{P}^1_k \setminus \{ 1 \}$, with the faces $\{0, \infty \}$ (see e.g. \cite{Bloch note}). This cubical version is also equivalent to the simplicial version (see e.g. \cite[Theorem 4.3, p.13]{Bloch note} for a spectral sequence argument based on $\mathbb{A}^1$-invariance), while we also have some corresponding advantages such as the product structure. 

Out of numerous properties of the motivic homology theory, one fundamental result is the localization theorem. Although this is important and useful in the theory for various purposes, it is the author's feeling that its proof has gained some notoriety of being difficult to digest. 

\medskip

The purpose of this article is to ease the situation a bit by giving a purely cubical proof of the localization theorem for the cubical version of the motivic homology theory, i.e. the cubical higher Chow groups. More specifically, it reads (see \S \ref{sec:cycles} for the recollections of the definitions):

\begin{thm}\label{thm:main}
Let $Y$ be a scheme of finite type over a field $k$. Let $U \subset Y$ be an open subscheme, and let $W:= Y \setminus U$ be the closed subset with the reduced induced closed subscheme structure. Let $W \overset{i}{\hookrightarrow} Y \overset{j}{\hookleftarrow} U$ be the respective closed and open immersions. Let $d \in \mathbb{Z}$.

Then for the exact sequence of the cubical higher Chow complexes
$$
0 \to z_d (W, \bullet)\overset{i_*}{\to} z_d (Y, \bullet)  \overset{j^*}{\to }z_d (U, \bullet),
$$
the cokernel $\mathcal{C}_{\bullet}:={\rm coker} (j^*)$ of the flat pull-back $j^*$ is acyclic. In particular
$$
z_d (W, \bullet) \to z_d (Y, \bullet) \to z_d (U, \bullet) \to z_d (W, \bullet) [1]
$$
gives a distinguished triangle in the derived category of complexes of abelian groups.
\end{thm}

The validity of this theorem is certain up to quasi-isomorphism, because it is known that the cubical and the simplicial versions of the higher Chow complexes are quasi-isomorphic to each other, while the localization theorem is known already for the simplicial version by \cite{Bloch moving}, \cite{Levine JPAA}, and \cite{Levine moving}.

Then why does the author want to prove it again? Decades after the arrival of the original simplicial version in S. Bloch \cite{Bloch HC}, many of the new generation of researchers in the field nowadays prefer to stick to the cubical version. 

On the other hand, the proof is given for the simplicial version only in \cite{Bloch moving}, \cite{Levine JPAA}, and \cite{Levine moving}, and the author witnessed a few frustrated younger researchers and students (including his younger self) who had significant difficulties in following the winding arguments there. So, it is meaningful to provide a road that is a bit less curvy, if not straight.

\medskip

The localization theorem for the simplicial version was first claimed for quasi-projective $k$-schemes in \cite[Theorem (3.3)]{Bloch HC}. Based on it, a higher generalization of the Grothendieck-Riemann-Roch theorem was obtained. A bit of gap was found, and a different idea (e.g. \cite{Levine K}) via $K$-theory was introduced to prove the Grothendieck-Riemann-Roch without resorting to the localization theorem. Meantime, the gap was filled up, and for the quasi-projective case the first proof was provided in  \cite{Bloch moving}. Later it was further generalized to finite type $k$-schemes when $k$ is a field or a DVR by \cite{Levine JPAA} and \cite{Levine moving}. 

\medskip

The arguments are nontrivial, and largely there are at least two parts that make it hard to digest those papers.

The first part is related to Hironaka's polyhedral game and its partial resolution by Spivakovsky \cite{Spivakovsky}. It was used by S. Bloch \cite{Bloch moving} (and its generalization in \cite{Levine JPAA}) to prove the existence of a suitable tower of face blow-ups for the cube $S_0:= \square^n$ that improves the proper intersection property of cycles with the faces. This is done purely in the cubical setting. 

The second part is on the need to mix both the simplicial and the cubical settings in the proof. The first part uses only the cubical setting, but since the higher Chow theory considered in \cite{Bloch moving} and \cite{Levine moving} was the simplicial version, this mixing was necessary in \emph{ibids}. This feature made the above papers harder to digest as a result.

\medskip

The specific aim of this paper is to provide a proof written purely in terms of the cubical version of the higher Chow complexes, by improving the above second part. To do so, we use a normalization theorem for $\mathcal{C}_{\bullet}$ that every cycle class can be represented by a ``normalized" cycle, i.e., it has the trivial codimension $1$ faces. New ingredients developed in detail in this article are pertaining to the bi-divisions and the cubical subdivisions. Roughly, we show that every normalized cycle that has the trivial codimension $1$ faces is equivalent to its bi-divisions and the cubical subdivisions. Once we have it, combined with the face blow-ups of cubes and a couple of auxiliary arguments, we can deduce the localization for the cubical higher Chow groups. The author believes this is psychologically and technically easier to follow.

\medskip

As an application, we give a concrete complex of flasque sheaves that is quasi-isomorphic to the sheafification $\BGHz_d (\bullet)_Y$ of cubical higher Chow complex, and we use it to show that the cycle class groups $\CH_d (Y, n)$ admit a description
$$
\CH_d (Y, n) = \mathbb{H}_{\rm Zar} ^{-n} (Y, \BGHz_d (\bullet)_Y)
$$
in terms of the hypercohomology of a complex of sheaves. In the middle, we deduce the Zariski Mayer-Vietoris property. While this is known already to experts, the author believes that the explicit argument from this perspective is meaningful.

\medskip

\textbf{Convention:} Here $k$ is a fixed arbitrary field. A $k$-scheme is assumed to be of finite type over $k$, but we do not assume anything extra unless we specifically say otherwise.

\medskip

\textbf{Acknowledgements:} JP was partially supported by the National Research Foundation of Korea (NRF) grant (2018R1A2B6002287) funded by the Korean government (Ministry of Science and ICT).

\section{Cubical higher Chow cycles}\label{sec:cycles}

In \S \ref{sec:cycles}, we recall some elementary definitions and results for the cubical higher Chow cycles that we need in this article.

\subsection{The cycle complexes} \label{subsec:cycles}
For the projective line $\mathbb{P}_k ^1$, let $\square_k := \mathbb{P}_k ^1 \setminus \{ 1 \}$. Let $\square^0 _k := \Spec (k)$. For $n \geq 1$, let $\square^n_k$ be the $n$-fold self product of $\square_k$ over $k$. Let $(y_1, \ldots, y_n) \in \square^n_k$ be the coordinates.

Define a face $F \subset \square_k ^n$ to be the closed subscheme given by a system of equations of the form $\{ y_{i_1} = \epsilon_1, \ldots, y_{i_s} = \epsilon_s\}$ for an increasing sequence of indices $1 \leq i_1 < \ldots < i_s \leq n$ and $\epsilon_j \in \{ 0, \infty \}$. In case of the empty set of equations, we take $F= \square_k ^n$. 
The codimension one face given by $\{ y_i = \epsilon\}$ is denoted by $F_i ^{\epsilon}$. 

\medskip

Let $Y$ be a $k$-scheme of finite type. The fiber product $Y \times_k \square_k ^n$ is often denoted by $\square_Y^n$. A \emph{face} $F_Y$ of $\square_Y^n$ is given by $Y \times _k F$ for a face $F \subset \square_k ^n$. Here, if $Y$ is equidimensional in addition, then so is $Y \times_k \square_k^n$.

\medskip

For a $k$-scheme $W$ of finite type and an integer $r \geq 0$, let $z_r (W)$ be the free abelian group on the set of $r$-dimensional integral closed subschemes of $W$. Its member is called an $r$-cycle on $W$.

\medskip

Here is the definition of the cubical version of the higher Chow cycles (cf. S. Bloch \cite{Bloch HC} and B. Totaro \cite{Totaro}):

\begin{defn}\label{defn:HCG}

Let $Y$ be a $k$-scheme of finite type. Let $n \geq 0$ and $d \in \mathbb{Z}$ be integers. Let $\un{z}_d (Y, n)$ be the subgroup of $z_{d+n} (\square_Y ^n)$ generated by the integral closed subschemes $Z \subset \square_{Y} ^n$ of dimension $d+n$, that intersect properly with $Y \times F$ for each face $F \subset \square_k ^n$. 

We may call the cycles in $\un{z}_d (Y, n)$ \emph{admissible} for simplicity. In case $Y$ is equidimensional, we also write $\un{z}^q (Y,n ) = \un{z}_d (Y, n)$ for $q:= \dim \ Y - d$.

\medskip

Let $n \geq 1$. We have the boundary operators $\partial: \un{z}_d (Y, n) \to \un{z}_d (Y, n-1)$ defined as follows. For $1 \leq i \leq n$ and $\epsilon \in \{ 0, \infty \}$, let $\iota_i ^{\epsilon}: \square_Y ^{n-1} \hookrightarrow \square_Y^n$ be the closed immersion given by the equation $\{y_i = \epsilon \}$. For each integral cycle $Z \in \un{z}_d (Y, n)$, define the face $\partial_i ^{\epsilon} (Z)$ to be the cycle associated to the scheme-theoretic intersection $(\iota_i ^{\epsilon})^* (Z) $. By definition, we have $\partial_i ^{\epsilon} (Z) \in \un{z} _d (Y, n-1)$.

 Let $\partial:= \sum_{i=1} ^n (-1)^i (\partial_i ^{\infty} - \partial_i ^0)$. %One checks readily that
 %$$
 %\partial _i ^{\epsilon} \partial_j ^{\epsilon} = \tuborg \partial_{j-1} ^{\epsilon} \partial_i ^{\epsilon}, & \mbox{ for } i < j, \\
 %\partial_j ^{\epsilon} \partial_i ^{\epsilon}, & \mbox{ for } i=j, \\
 %\partial_j ^{\epsilon} \partial_{i+1} ^{\epsilon}, & \mbox{ for } i > j.\sluttuborg
 %$$
 From the usual cubical formalism, one checks that $\partial \circ \partial = 0$ and that we have a cubical abelian group $(\un{n} \mapsto \un{z} _d (Y, n))$, where $\un{n}:= \{ 0, 1, \ldots, n \}$. %which means it defines a functor $\textbf{Cube} ^{\rm op} \to ({\rm Ab})$. For relevant generalities on cubical abelian groups and the related formalisms, see M. Levine \cite[\S 1]{Levine SM}.

Let $z_d (Y, \bullet)$ be the associated non-degenerate complex, i.e. the quotient complex
\begin{equation}\label{eqn:Deg}
z_d (Y, \bullet) = \un{z}_d (Y, \bullet)/ \un{z}_d (Y, \bullet)_{\rm degn},
\end{equation}
where the subgroup $\un{z}_d (Y, \bullet)_{\rm degn}$ consists of the degenerate cycles; this group is generated by the pull-backs of cycles via the coordinate projections $\square_{Y} ^n \to \square_{Y} ^{n-1}$ that forget one of the coordinates. This $z_d (Y, \bullet)$ is called the cubical higher Chow complex.

Define the homology group $\CH_d (Y, n):= {\rm H}_n (z_d (Y, \bullet))$, and it is called the cubical higher Chow group. 

If $Y$ is equidimensional, we similarly define $z^q (Y, \bullet)$ and $\CH^q (Y, n)$, in terms of the codimension $q \geq 0$.
\qed
\end{defn}

\subsection{Normalization of some cubical groups}\label{sec:normalization}  

In \S \ref{sec:normalization} we recall the normalization theorem for the cubical higher Chow complex. Here, the word normalization originates from the usage of the word in the Dold-Kan correspondence for simplicial abelian groups. See J. P. May \cite[Theorem 22.1, p.94]{May}.

For a cubical abelian group, a normalization theorem may not hold in general, but M. Levine \cite[Lemma 1.6]{Levine SM} had identified a sufficient condition for it. It says an \emph{extended} cubical abelian group always has this property. Recall (\cite[\S 1]{Levine SM})  that an extended cubical abelian group is a functor $\textbf{ECube}^{\rm op} \to ({\rm Ab})$, where $\textbf{ECube}$ is the smallest symmetric monoidal subcategory of the category of sets containing the category $\textbf{Cube}$, and the morphism $\mu: \un{2} \to \un{1}$.

\medskip

For the cubical cycle groups, we know the following:

\begin{thm}\label{thm:normalization0}
Let $Y$ be a $k$-scheme of finite type. For the cubical higher Chow complex $z_d (Y, \bullet)$, let $z_{d}^N (Y, n) \subset z_d (Y, n)$ be the subgroup generated by cycles $Z$ such that $\partial_i ^0 (Z) = 0$ for all $1 \leq i \leq n $ and $\partial_i ^{\infty} (Z) = 0$ for $2 \leq i \leq n$.

With the ``last" face $\partial_1 ^{\infty}$, this gives the {normalized subcomplex} $(z_{d}^N (Y, \bullet), \partial_1 ^{\infty}) \hookrightarrow (z_d (Y, \bullet), \partial)$, and this inclusion is a quasi-isomorphism. 
\end{thm}

This was proven in Bloch's notes \cite[Theorem 4.4.2]{Bloch note}. In the published literature, one can find an essentially identical proof in M. Li \cite[Theorem 2.6]{Li}. For a pair $(Y, D)$ consisting of a smooth $k$-scheme and an effective Cartier divisor, similar results are known for the higher Chow cycles $z_d (Y|D, \bullet)$ with modulus. See Krishna-Park \cite[Theorem 3.2]{KP DM}.

\medskip

We need the following version for $\mathcal{C}_{\bullet}$ of Theorem \ref{thm:main} as well, which essentially follows from the arguments of Theorem \ref{thm:normalization0}.

\begin{thm}\label{thm:normalization}
Let $Y$ be a $k$-scheme of finite type. Let $U \subset Y$ be a nonempty open subscheme, and let
$$
\mathcal{C}_{\bullet}:= {\rm coker} (\rho^Y_U: z_d (Y, \bullet) \to z_d (U, \bullet)).
$$

Let $\mathcal{C}_{\bullet}^N \subset \mathcal{C}_{\bullet}$ be the normalized subcomplex, i.e. $\mathcal{C}_{n}^N$ consists of the images of cycle classes $Z \in z_d (U, n)$ such that $\partial_i ^{0} Z \in \rho^Y _U (z_d (Y, n-1))$ for each $1 \leq i \leq n$ and $\partial_i ^{\infty} Z \in \rho^Y _U (z_d (Y, n-1))$ for all $2 \leq i \leq n$

Then the inclusion $(\mathcal{C}_{\bullet} ^N , \partial_1 ^{\infty}) \hookrightarrow (\mathcal{C}_{\bullet}, \partial)$ is a quasi-isomorphism.
\end{thm}

\begin{proof}
This essentially comes from the argument of Theorem \ref{thm:normalization0}. We sketch the proof.

For the proof, via the automorphism $y \mapsto \frac{y}{y-1}$ of $\mathbb{P}^1_k$, we may replace $\square_k ^1$ with $\mathbb{A}^1$ with the faces $\{ 1, 0 \}$. To stress that our space and the faces are different, and when we use this setting, we will use the notation $\square_{\psi}^1= \mathbb{A}^1$.

By \cite[Lemma 1.6]{Levine SM}, it is enough to show that $(\un{n}\mapsto \mathcal{C}_n)$ is an extended cubical abelian group. The morphism $\un{2} \to \un{1}$ in \textbf{ECube} corresponds to the morphism $\mu: \square_{\psi}^2 \to \square_{\psi}^1$ given by $(y_1, y_2) \mapsto y_1 y_2$.

 Since $(\un{n} \mapsto z_d (Y, n))$ is an extended cubical abelian group (see \cite[Theorem 4.4.2]{Bloch note}, \cite[Theorem 2.6]{Li}), we already know that for the morphism $\mu: \square_{\psi} ^2 \to  \square_{\psi}^1$, we have $\mu^*: z_d (-, 1) \to z_d (-, 2)$ for $(-) = Y, U$. They form the commutative diagram
$$
\xymatrix{ z_d (Y, 1) \ar[d] ^{\mu^*} \ar[r] ^{\rho^Y_U} & z_d (U, 1) \ar[d] ^{\mu^*} \\
z_d (Y, 2) \ar[r] ^{\rho^Y_U} & z_d (U, 2).}
$$
Thus we deduce the homomorphism $\mu^*: \mathcal{C}_1 \to \mathcal{C}_2$. Since a general morphism in \textbf{ECube} is generated by the original morphisms of \textbf{Cube} and the above $\mu$, we deduce the theorem.
\end{proof}

\subsection{Pull-backs and push-forwards}

We recall the following basic results. Their proofs are immediately deduced by the corresponding assertions for the simplicial versions in \cite[Proposition (1.3), Corollary (1.4)]{Bloch HC}, \emph{mutatis mutandis}.

\begin{prop}
\begin{enumerate}
\item Let $f: Y_1 \to Y_2$ be a proper morphism of $k$-schemes of finite type. Then we have the associated push-forward morphism $f_*: z_d  (Y_1, \bullet) \to z_d (Y_2, \bullet)$ of complexes of abelian groups.
\item Let $f: Y_1 \to Y_2$ be a flat morphism of $k$-schemes of finite type of relative dimension $c$. Then we have the associated pull-back morphism $f^*: z_d  (Y_2, \bullet) \to z_{d+c} (Y_1, \bullet)$ of complexes of abelian groups.
\end{enumerate}
\end{prop}

\begin{cor}\label{cor:pbpf degree}
Let $Y$ be a $k$-scheme of finite type. Let $k \hookrightarrow k'$ be a finite extension of fields and let $\pi: Y_{k'} \to Y_k$ be the induced base change map, which is both flat and proper. Then the composite
$$
z_d (Y_k , \bullet) \overset{\pi^*}{\to} z_d (Y_{k'}, \bullet) \overset{\pi_*}{\to} z_d (Y_k, \bullet)
$$
is the multiplication by $[k': k]$.
\end{cor}

\section{Cubical subdivisions and the localization theorem}\label{sec:localization}

The goal of \S \ref{sec:localization} is to prove the following Theorem \ref{thm:localization}, which is Theorem \ref{thm:main} .

\begin{thm}\label{thm:localization}
Let $Y$ be a $k$-scheme of finite type. Let $U \subset Y$ be an open subscheme.

Consider the restriction morphism
\begin{equation}\label{eqn:localization}
 \rho ^Y_U: z_d (Y, \bullet) \to z_d (U, \bullet).
\end{equation}

Then $\mathcal{C}_{\bullet}:= {\rm coker} (\rho^Y_U)$ is acyclic.
\end{thm}

Our focus in \S \ref{sec:localization} is on the cokernel of \eqref{eqn:localization}. The kernel of \eqref{eqn:localization} plays roles in \S \ref{sec:sheaf}.

\medskip

Part of the argument needs one of the techniques, ``blowing-up of faces of a cube" of S. Bloch \cite{Bloch moving}, and its systematic reorganization and generalization by M. Levine \cite{Levine JPAA} and \cite{Levine moving}. But there is a noteworthy difference in our arguments compared to \emph{ibids.}

\medskip

One of aspects there was a need to mix both the simplicial and the cubical settings. Since the face blow-ups are done in the cubical setting, while the simplicial version of the complexes was used there, the resulting mixing was needed. 

Here, we use the cubical version only. Instead we complement the arguments with the normalization theorem (Theorem \ref{thm:normalization}), and new ingredients of the explicit cubical subdivision calculations given in \S \ref{sec:cubical subdivision}.

Some materials from S. Bloch \cite{Bloch moving} and M. Levine \cite{Levine JPAA}, \cite{Levine moving} are recalled and sketched in \S \ref{sec:moving by blow-up}. There is no new material in \S \ref{sec:moving by blow-up}. The proof of Theorem \ref{thm:localization} is finished in \S \ref{sec:proof localization}.

\subsection{Cubical subdivision}\label{sec:cubical subdivision}
For a while, we suppose $k$ is an infinite field. It will be dropped in \S \ref{sec:proof localization}.

\medskip

The purpose of \S \ref{sec:cubical subdivision} is to study cubical subdivisions of cubical higher Chow cycles over a $k$-scheme $Y$, and to prove that a cycle (with a normalization condition, i.e., all codimension $1$ faces vanish) is equivalent to its cubical subdivision modulo a boundary. Its precise statement is given later in Proposition \ref{prop:cub sd}. We defer it as we need a couple of definitions, though below we illustrate roughly what we are up to do. Let $n \geq 1$.

\medskip

For a $k$-rational point $c = (c_1, \ldots, c_n) \in \square^n_k$ not meeting any faces, topologically imagine the cubical subdivision given by the point $c$ as depicted in the Figure 1.

\begin{center}{Figure 1}
\begin{tikzpicture}[line cap=round,line join=round,>=triangle 45,x=1.0cm,y=1.0cm]
\clip(-0.3,-0.3) rectangle (3.1,3.1);
\fill[line width=2.pt,fill opacity=0.10000000149011612] (0.,3.) -- (0.,0.) -- (3.,0.) -- (3.,3.) -- cycle;
\draw [line width=2.pt] (0.,3.)-- (0.,0.);
\draw [line width=2.pt] (0.,0.)-- (3.,0.);
\draw [line width=2.pt] (3.,0.)-- (3.,3.);
\draw [line width=2.pt] (3.,3.)-- (0.,3.);
\draw [line width=2.pt,dash pattern=on 5pt off 5pt] (1.56,4.)-- (1.56,0.);
\draw [line width=2.pt,dash pattern=on 5pt off 5pt] (0.,1.76)-- (4.,1.76);
\begin{scriptsize}
\draw (1.75,-0.19) node {$c_1$};
\draw (-0.19,1.65) node {$c_2$};
\end{scriptsize}
\end{tikzpicture}
\end{center}

This picture is just for our intuition only and one should keep in mind that our algebraic ``cubes'' are still affine spaces. Literally cutting along the dotted lines is neither possible nor a good thing to perform in algebraic geometry.

What we really do is to make copies of the original cycle $Z$ on $Y \times \square^n_k$ as many as the number of vertices, and to regard a scaling and an involution of each copy as a ``piece'' of the subdivision, equipped with a new set of faces given by some of the original ``distinguished" codimension $1$-faces and the new ``internal" codimension $1$-faces given by the hyperplanes $\{y_i = c_i\}$. The magic is that under the normalization theorem, Theorem \ref{thm:normalization}, this cubical subdivision is equivalent to the original one, up to the boundaries of some concrete cycles.

\medskip

The basic story of the proof of Proposition \ref{prop:cub sd} is illustrated in the Figure 2. More precisely, given a normalized cycle ${Z}$, i.e. a cycle all of whose codimension $1$ faces are trivial, we show that for a general $c$, there is a sequence of ``bi-divisions"
$$
{Z} \overset{ \phi_{c, 1}}{\equiv} \delta_{c_1,1} ({Z}) \overset{  \phi_{c, 2}}{\equiv} \delta_{c_2,2} \delta_{c_1,1} ({Z}) \equiv \cdots \overset{  \phi_{c, n}}{\equiv} \delta_{c_n,n} \cdots \delta_{c_1,1} ({Z}) = {\rm sd}_{c} ({Z}),
$$
where $\delta_{c_i,i}$ is the operation of ``dividing" the coordinate $y_i$ into two pieces at $y_i = c_i$ (i.e. making two copies and scaling suitably), and $\equiv$ means an equivalence modulo the boundary $\phi_{c, i}$ of a cycle. 
We make the above description precise in what follows.

\begin{center}{Figure 2}
\begin{tikzpicture}[line cap=round,line join=round,>=triangle 45,x=1.0cm,y=1.0cm]
\clip(-0.1,-0.1) rectangle (12.9,3.3);
\fill[line width=2.pt,fill opacity=0.10000000149011612] (0.,3.) -- (0.,0.) -- (3.,0.) -- (3.,3.) -- cycle;
\fill[line width=2.pt,fill opacity=0.10000000149011612] (5.,3.) -- (5.,0.) -- (6.45,0.) -- (6.45,3.) -- cycle;
\fill[line width=2.pt,fill opacity=0.10000000149011612] (7.,3.) -- (7.,0.) -- (8.,0.) -- (8.,3.) -- cycle;
\fill[line width=2.pt,fill opacity=0.10000000149011612] (10.,3.) -- (10.,2.) -- (11.45,2) -- (11.45,3.) -- cycle;
\fill[line width=2.pt,fill opacity=0.10000000149011612] (10.,1.7) -- (11.45,1.7) -- (11.45,0.) -- (10.,0.) -- cycle;
\fill[line width=2.pt,fill opacity=0.10000000149011612] (11.85,3.) -- (11.85,2) -- (12.7,2) -- (12.7,3) -- cycle;
\fill[line width=2.pt,fill opacity=0.10000000149011612] (11.85,0.) -- (11.85,1.7) -- (12.7,1.7) -- (12.7,0.) -- cycle;
\draw [line width=2.pt] (0.,3.)-- (0.,0.);
\draw [line width=2.pt] (0.,0.)-- (3.,0.);
\draw [line width=2.pt] (3.,0.)-- (3.,3.);
\draw [line width=2.pt] (3.,3.)-- (0.,3.);
\draw [line width=2.pt] (5.,3.)-- (5.,0.);
\draw [line width=2.pt] (5.,0.)-- (6.45,0.);
\draw [line width=2.pt] (6.45,0.)-- (6.45,3.);
\draw [line width=2.pt] (6.45,3.)-- (5.,3.);
\draw [line width=2.pt] (7.,3.)-- (7.,0.);
\draw [line width=2.pt] (7.,0.)-- (8.,0.);
\draw [line width=2.pt] (8.,0.)-- (8.,3.);
\draw [line width=2.pt] (8.,3.)-- (7.,3.);
\draw [line width=2.pt] (10.,3.)-- (10.,2.);
\draw [line width=2.pt] (10.,2.)-- (11.45,2);
\draw [line width=2.pt] (11.45,2)-- (11.45,3.);
\draw [line width=2.pt] (11.45,3.)-- (10.,3.);
\draw [line width=2.pt] (10.,1.7)-- (11.45,1.7);
\draw [line width=2.pt] (11.45,1.7)-- (11.45,0.);
\draw [line width=2.pt] (11.45,0.)-- (10.,0.);
\draw [line width=2.pt] (10.,0.)-- (10.,1.7);
\draw [line width=2.pt] (11.85,3.)-- (11.85,2);
\draw [line width=2.pt] (11.85,2)-- (12.7,2);
\draw [line width=2.pt] (12.7,2)-- (12.7,3);
\draw [line width=2.pt] (12.7,3)-- (11.86,3.);
\draw [line width=2.pt] (11.85,0.)-- (11.85,1.7);
\draw [line width=2.pt] (11.85,1.7)-- (12.7,1.7);
\draw [line width=2.pt] (12.7,1.7)-- (12.7,0.);
\draw [line width=2.pt] (12.7,0.)-- (11.85,0.);
\draw [line width=2.pt,dash pattern=on 5pt off 5pt] (1.8,3.)-- (1.8,0.);
\draw [line width=2.pt,dash pattern=on 5pt off 5pt] (0.,2.)-- (3.,2.);
\draw [line width=2.pt,dash pattern=on 5pt off 5pt] (5.,2.)-- (6.45,2.);
\draw [line width=2.pt,dash pattern=on 5pt off 5pt] (7.,2.)-- (8.,2.);
\draw [->,line width=2.pt] (3.48,1.22) -- (4.62,1.22);
\draw [->,line width=2.pt] (8.48,1.22) -- (9.6,1.22);
\end{tikzpicture}
\end{center}

\subsubsection{Bi-division}\label{sec:bi-division}
For simplicity, we use $\square_{\psi} := \mathbb{A}^1$ again. Under the identification $\square \simeq \square_{\psi}$ given by the automorphism $y \mapsto \frac{y}{y-1}$ of $\mathbb{P}_k^1$, the faces $\{ 0, 1 \} \subset \square_{\psi}$ correspond to the faces $\{ 0, \infty \} \subset \square$. A face of $\square_{\psi} ^n$ is defined similarly using $\{ 0, 1 \} \subset \square_{\psi}^1$. A \emph{vertex} of $\square^n_{\psi} $ is a face of dimension $0$; it is of the form $v = (v_1, \ldots, v_n)$, where $v_i \in \{ 0, 1 \}$. % each $v_i $ is either $ 0$ or $1$. 
Define ({cf. S. Bloch \cite[Definition (3.1.1)]{Bloch moving}}) its associated sign 
\begin{equation}\label{eqn:sgn_orientation}
\varepsilon (v) := (-1)^m,
\end{equation}
where $m:= |B(v)|$ with $ B(v):= \{ i \ | \ v_i = 1 \} $. 

Until the end of \S \ref{sec:localization}, we use the above cubes $\square_{\psi}^n$ and the resulting faces for the cycles in $z_* (Y, n)$.

\begin{defn}\label{defn:involution} 
For each $1 \leq j \leq n$, the $j$-th involution $\tau_j : \square_{\psi} ^n \to \square_{\psi}^n$ is
$$
y_i \mapsto \tuborg y_i, & \mbox{ if } i \not = j, \\
1- y_j, & \mbox{ if } i = j.\sluttuborg
$$

For a given $c \in k \setminus \{ 0, 1 \}$, the \emph{$j$-th scaling} $ \sigma_{c, j}: \square_{\psi} ^n \to \square_{\psi} ^n$ by $c$  is
$$
y_i \mapsto \tuborg y_i, & \mbox{ if } i \not = j, \\
c y_j, & \mbox{ if } i=j.\sluttuborg
$$ 
Since these are isomorphisms, they do give pull-backs of cycles on $Y \times \square_{\psi}^n$.
\qed
\end{defn}

\begin{lem}\label{lem:involution adm}
Let $Y$ be a $k$-scheme of finite type. 
Let $1 \leq j \leq n$. Then for all ${Z} \in z_d (Y, n)$, we have $\tau_j ^* ({Z}) \in z_d (Y, n)$.
\end{lem}

\begin{proof}
Since the involution $\tau_j$ of Definition \ref{defn:involution} interchanges the codimension $1$ faces $\{ y_j = 0\}$ and $\{ y_j = 1 \}$ while leaving all other codimension $1$ faces fixed, the property of proper intersection with the faces for ${Z}$ immediately implies the property for $\tau_j ^* ({Z})$. The lemma follows.
\end{proof}

\begin{lem}\label{lem:scaling adm}
Let $Y$ be a $k$-scheme of finite type. 
Let $1 \leq j \leq n$. Let ${Z} \in z_d (Y, n)$. Then for a general $k$-rational point $c \in \square_{\psi}^1 \setminus \{ 0,1 \}$, both $\sigma_{c, j} ^* ({Z})$ and $\sigma_{1-c, j}^* \tau_j ^* ({Z}) \in z_d (Y, n)$.
\end{lem}

\begin{proof}
We may assume $Z$ is integral. We prove that for a dense open subset $U \subset \square_{\psi}^1 \setminus \{ 0, 1 \}$, for each $c \in U (k)$, the cycles $\sigma_{c,j}^* ({Z})$ and $\sigma_{1-c, j}^* \tau_j ^* ({Z})$ belong to $z_d (Y, n)$.

Since the maps $\sigma_{c,j}$ and $\tau_j \circ \sigma_{1-c, j}$ are both isomorphisms, the cycles $\sigma_{c,j}^* ({Z})$ and $\sigma_{1-c, j}^* \tau_j ^* ({Z})$ intersect properly with all closed subschemes of the form $Y \times F$ for faces $F  \subset \square_{\psi}^n$ if and only if ${Z}$ intersects properly with all closed subschemes of the form $Y \times F'$, where $F' \subset \square_{\psi}^n$ is a closed subscheme obtained as a finite intersection of the divisors chosen from
\begin{itemize}
\item the original ``distinguished" divisors $\{ y_i = \epsilon \}$ for $1 \leq i \leq n$ and $\epsilon \in \{ 0, 1 \}$, and
\item  the new ``internal" divisor $F(j,c):=\{y_j = c \}$.
\end{itemize}

If $F'$ is entirely chosen from the original distinguished divisors, the proper intersection is already given by that of ${Z}$. It remains to see what happens if $F'$ is of the form $ F'' \cap F (j,c)$, where $F''$ is a face from the original distinguished divisors.

Observe first that for some $c \in k \setminus \{ 0, 1 \}$, the divisor $Y \times F (j,c)$ \emph{does not} intersect the cycle ${Z}\cap ( Y \times F'' )$ properly if and only if $Y \times F (j,c) $ contains a component of ${Z} \cap (Y \times F'')$. This could happen with at most finitely many $c$.

Thus for a nonempty open subset $U \subset \square_{\psi} ^1 \setminus \{ 0, 1 \}$ and for each $k$-rational point $c \in U$ (here we need that $|k|=\infty$ to guarantee that $U (k) \not = \emptyset$), the cycle ${Z}$ intersects $Y \times F'$ properly for any such $F' \subset \square_{\psi}^n$ considered in the above. This shows the proper intersection property for both $\sigma_{c, j} ^* ({Z})$ and $\sigma_{1-c, j}^* \tau_j ^* ({Z})$ for such $c$.
This proves the lemma.
\end{proof}

\begin{defn}\label{defn:bidivision delta}
Let $c \in k \setminus \{ 0, 1 \}$. Let $1 \leq i \leq n$. Define the \emph{bi-division morphism at $y_i = c$} to be the finite formal $\mathbb{Z}$-linear sum of morphisms
$$
\delta_{c, i}:= \sigma_{c, i} - \tau_i \circ \sigma_{1-c, i}.
$$

Let $Y$ be a $k$-scheme of finite type.
For a given cycle ${Z}$ on $Y \times \square_{\psi} ^n$, its \emph{bi-division at $y_i = c$} is the sum of the pull-backs of the cycles
$$
\delta_{c,i}^* ({Z}):= \sigma_{c, i} ^* ({Z}) - \sigma_{1-c, i} ^* \tau_i ^*({Z}).
$$
For simplicity, we just write $\delta_{c,i} ({Z})$ instead of $\delta_{c,i}^* ({Z})$. \qed
\end{defn}

\begin{cor}\label{lem:bidivision adm}
Let $Y$ be a $k$-scheme of finite type. Let $1 \leq i \leq n$. Let ${Z} \in z_d (Y, n)$. 

Then for a general $k$-rational point $c \in \square_{\psi} ^1 \setminus \{ \mbox{faces}\}$, we have $\delta_{c,i} ({Z}) \in z_d (Y, n)$.
\end{cor}

\begin{proof}
This follows from Lemma \ref{lem:scaling adm} and Definition \ref{defn:bidivision delta}.
\end{proof}

\subsubsection{Homotopy for bi-division}\label{sec:homo bidivision}
Under some assumptions on a cycle ${Z}$ on $Y \times \square_{\psi}^n$, we want to relate ${Z}$ to its bi-division $\delta_{c, i} ({Z})$ at $y_i = c_i$ via an explicit homotopy cycle. This is done in steps. The construction of the homotopy relies on:

\begin{defn}\label{defn:eta}
Let $c \in k\setminus \{ 0, 1 \}$. Define morphisms $\eta ^0_c $ and $\eta ^1_c : \square_{\psi} ^2 \to \square_{\psi} ^1$ by
$$
\tuborg \eta_c ^0 (y, z) = (1 - (1-c)(1-z))y, \\
\eta_c ^1 (y, z) = 1 - (1-c) y (1-z).\sluttuborg
$$
\end{defn}

\begin{lem}\label{lem:eta bdry} The following properties hold:
\begin{enumerate}
\item $\tuborg \eta_c ^0 (y, 0) = cy,  \ \ \ \ \  & \eta_c ^0 (y, 1) = y,  \\
\eta_c ^0 (0, z) = 0, & \eta_c ^0 (1, z) = 1 - (1-c)(1-z). \sluttuborg$

\item $\tuborg \eta_c ^1 (y, 0) = 1- (1-c) y, \ \ \ \ & \eta_c ^1 (y, 1) = 1, \\
\eta_c ^1 (0, z) = 1, & \eta_c ^1 (1, z) = 1- (1-c)(1-z).\sluttuborg$

\end{enumerate}

\end{lem}

\begin{proof}
This is straightforward from Definition \ref{defn:eta}.
\end{proof}

\begin{defn}
Let $1 \leq i \leq n$ and let $\ell = 0, 1$. For a given $k$-rational point $c= (c_1, \ldots, c_n) \in \square_{\psi}^n \setminus \{ \mbox{\rm faces} \}$, define the morphism
$$
H_{(\ell), c, i} ^{n+1} : \square_{\psi} ^{n+1} \to \square_{\psi} ^n, \ \ \ (y_1, \ldots, y_n, y_{n+1})\mapsto(m_1, \ldots, m_n), \ \mbox{where}
$$ 
$$
m_j: = \tuborg y_j, & \mbox{ if } j \not = i,\\
\eta_{c_i} ^{\ell} (y_i, y_{n+1}), & \mbox{ if } j = i, \sluttuborg
$$
 for $1 \leq j \leq n$. \qed
\end{defn}

\begin{lem}\label{lem:Hell flat}
The morphisms $H_{(\ell), c, i} ^{n+1} : \square_{\psi} ^{n+1} \to \square_{\psi} ^{n}$ are flat.
\end{lem}

\begin{proof}
The morphism $H_{(\ell), c,i}^{n+1}$ is a permutation of ${\rm Id}_{\square_{\psi} ^{n-1}} \times \eta_{c_i} ^{\ell}$. So, it is enough to check that the morphisms $\eta_c ^{\ell}: \square_{\psi} ^2 \to \square_{\psi} ^1$ in Definition \ref{defn:eta} are flat. 

For $\ell=0$, the map $\eta_c ^0$ is equal to the composite
\begin{equation}\label{eqn:H flat 0}
\square_{\psi}^2 \overset{\tau_2}{\to} \square_{\psi}^2 \overset{\sigma_{1-c, 2}}{\longrightarrow} \square_{\psi}^2 \overset{\tau_2}{ \to} \square_{\psi} ^2\overset{\mu}{ \to} \square_{\psi}^1,
\end{equation}
where $\mu (y, z) = yz$. All morphisms in \eqref{eqn:H flat 0} other than $\mu$ are isomorphisms, while $\mu$ is flat. Thus $\eta_c ^0$ is flat.

Similarly for $\ell=1$, the map $\eta_c ^1$ is equal to the composite
\begin{equation}\label{eqn:H flat 1}
\square_{\psi}^2 \overset{\tau_2}{\to} \square_{\psi}^2 \overset{\sigma_{1-c, 1}}{\longrightarrow} \square_{\psi} ^2 \overset{\mu}{\to} \square_{\psi} ^1 \overset{\tau_1}{\to} \square_{\psi}^1,
\end{equation}
where all morphisms in \eqref{eqn:H flat 1} other than $\mu$ are isomorphisms, while $\mu$ is flat. Thus $\eta_c ^1$ is flat.
\end{proof}

\begin{lem}\label{lem:Hell adm}
Let $Y$ be a $k$-scheme of finite type. 
Let ${Z} \in z_d (Y, n)$. 

Then the flat pull-back $(H_{(\ell), c, i} ^{n+1} )^* ({Z})$ exists as a cycle on $Y \times \square_{\psi}^{n+1}$, and for a general $k$-rational point $c \in \square_{\psi}^n \setminus \{ \mbox{faces} \}$, we have $(H_{(\ell), c, i} ^{n+1} )^* ({Z}) \in z_d (Y, n+1)$.
\end{lem}

\begin{proof}

The map $H_{(\ell), c,i}^{n+1}$ is flat by Lemma \ref{lem:Hell flat}, so that the flat pull-back $(H_{(\ell), c, i} ^{n+1} )^* ({Z})$ exists in ${z} _{d+n+1} (Y \times \square_{\psi} ^{n+1})$.

 It remains to check that it belongs to $z_d (Y, n+1)$ for a general $c$. Here, by \eqref{eqn:H flat 0} and \eqref{eqn:H flat 1} in the proof of Lemma \ref{lem:Hell flat}, the morphism $\eta_{c_i} ^{\ell}$ is a composite of $\mu$, involutions and a scaling isomorphism in Definition \ref{defn:involution}. 
 
 We already know that $\mu^*$ sends admissible cycles to admissible ones (\cite[Theorem 4.4.2]{Bloch note}, \cite[Theorem 2.6]{Li}, or proof of Theorem \ref{thm:normalization}). For the other two morphisms, we have Lemmas \ref{lem:involution adm}, and \ref{lem:scaling adm}. Thus the statement follows. Here a generic choice of $c$ is required to apply the Lemma \ref{lem:scaling adm}. Thus $(H_{(\ell), c, i} ^{n+1} )^* ({Z}) \in z_d (Y, n+1)$.
\end{proof}

We inspect the faces of $(H_{(\ell), c, i} ^{n+1} )^* ({Z})$. We introduce an auxiliary notation:

\begin{defn}For $c= (c_1, \ldots, c_n)$ and for $1 \leq j \leq n$, let
$$
c/j:= (c_1, \ldots, c_{j-1}, \check{c}_j, c_{j+1}, \ldots, c_n),
$$
where $\check{c}_j$ means $c_j$ is omitted.\qed
\end{defn}

The following is rather straightforward but tedious to check. We record the details to dispel potential concerns of the reader, and of the author himself as well:

\begin{lem}\label{lem:bidivision homotopy}
For $1 \leq i \leq n$, $1 \leq j \leq n+1$, $\epsilon \in \{ 0, 1 \}$, $\ell \in \{ 0, 1 \}$ and a $k$-rational point $c \in \square_{\psi} ^n \setminus \{ \mbox{\rm faces} \}$, we have the following:
\begin{enumerate}
\item For $j \not = i$ with $1 \leq j \leq n$, we have
$$
H_{(\ell), c, i}^{n+1} \circ \iota_j ^{\epsilon} = \tuborg \iota_j ^{\epsilon} \circ H_{ (\ell), c/j, i-1} ^n, & \mbox{ if } j < i, \\
\iota_j ^{\epsilon} \circ H_{(\ell), c/j, i} ^n , & \mbox{ if } j > i.\sluttuborg
$$
\item Now let $j=i$. The composite $H_{(\ell), c, i} ^{n+1} \circ \iota_i ^0$ is equal to the morphism
$$
\square_{\psi} ^n \to \square_{\psi}^n, \ \ (y_1, \ldots, y_n) \mapsto (y_1, \ldots, y_{i-1}, \ell, y_i, \ldots, y_{n-1}),
$$
where the coordinate $y_n$ is ignored.

\item We have the equality $H_{(0), c , i} ^{n+1} \circ \iota_i ^1 = H_{(1), c, i} ^{n+1} \circ \iota_i ^1$. 

\item $H_{(0), c, i} ^{n+1} \circ \iota_{n+1} ^{\epsilon} = \tuborg \sigma_{c_i, i}, & \mbox{ if } \epsilon = 0, \\
{\rm Id}, & \mbox{ if } \epsilon = 1.\sluttuborg$

\item $H_{(1), c, i} ^{n+1} \circ \iota_{n+1} ^{0} =  \tau_i \circ \sigma_{1-c_i, i}$, while $H_{(1), c, i} ^{n+1} \circ \iota_{n+1} ^{1}$ is equal to the morphism
$$
\square_{\psi}^n \to \square_{\psi}^n, \ \ (y_1, \ldots, y_n) \mapsto (y_1, \ldots, y_{i-1}, 1, y_{i+1}, \ldots, y_n),
$$
where the coordinate $y_i$ is ignored.
\end{enumerate}

\end{lem}

\begin{proof}
For $(y_1, \ldots, y_n ) \in \square_{\psi}^n$, we have
\begin{equation}\label{eqn:bidivision homotopy 0}
H_{(\ell), c, i}^{n+1} \circ \iota_j ^{\epsilon} (y_1, \ldots, y_n) = H_{(\ell), c, i}^{n+1} ( y_1, \ldots, y_{j-1}, \epsilon, y_j, \ldots, y_n),
\end{equation}
which we call $(m_1, \ldots, m_n).$

\medskip

(1) \emph{Case 1-1:} Suppose $j \leq i-2$. Then \eqref{eqn:bidivision homotopy 0} is
\begin{equation}\label{eqn:bdh 1-1}
(y_1, \ldots, y_{j-1}, \epsilon, y_j, \ldots, y_{i-2}, \eta_{c_i} ^{\ell} (y_{i-1}, y_n), y_i, \ldots, y_{n-1}).
\end{equation}

\medskip

\emph{Case 1-2:} Suppose $j= i-1$. Then \eqref{eqn:bidivision homotopy 0} is
\begin{equation}\label{eqn:bdh 1-2}
(y_1, \ldots, y_{i-2}, \epsilon, \eta_{c_i} ^{\ell} (y_{i-1}, y_n), y_i, \ldots, y_{n-1}).
\end{equation}

\medskip

In both of \eqref{eqn:bdh 1-1} and \eqref{eqn:bdh 1-2}, the expressions are equal to $\iota_j ^{\epsilon} \circ H_{(\ell), c/j, i-1} ^n (y_1, \ldots, y_n)$. This answers (1) when $j<i$.

\medskip

\emph{Case 2-1:} Suppose $i \leq j-2$. Then \eqref{eqn:bidivision homotopy 0} is
\begin{equation}\label{eqn:bdh 2-1}
(y_1, \ldots, y_{i-1}, \eta_{c_i} ^{\ell} (y_i, y_n), y_{i+1}, \ldots, y_{j-1}, \epsilon, y_j, \ldots, y_{n-1}).
\end{equation}

\medskip

\emph{Case 2-2:} Suppose $i=j-1$. Then \eqref{eqn:bidivision homotopy 0} is
\begin{equation}\label{eqn:bdh 2-2}
(y_1, \ldots, y_{i-1}, \eta_{c_i} ^{\ell} (y_i, y_n), \epsilon, y_{i+1}, \ldots, y_{n-1}).
\end{equation}

\medskip

In both of \eqref{eqn:bdh 2-1} and \eqref{eqn:bdh 2-2}, the expressions are equal to $\iota_j ^{\epsilon} \circ H_{(\ell), c/j, i} ^n (y_1,  \ldots, y_n)$. This answers (1) when $j>i$.

\medskip

(2) Here we have $j=i$ and $\epsilon = 0$ in \eqref{eqn:bidivision homotopy 0} so that 
$$
(m_1, \ldots, m_n) = (y_1, \ldots, y_{i-1}, \eta_{c_i} ^{\ell} (0, y_n), y_i, \ldots, y_{n-1}).
$$
By Lemma \ref{lem:eta bdry}, we have $\eta_{c_i} ^{\ell} (0, y_n) = \ell$. This shows (2). 

\medskip

(3) Here we have $j=i$ and $\epsilon = 1$ in \eqref{eqn:bidivision homotopy 0} so that 
\begin{equation}\label{eqn:ulala1}
(m_1, \ldots, m_n) = (y_1, \ldots, y_{i-1}, \eta_{c_i} ^{\ell} (1, y_n), y_i, \ldots, y_{n-1}).
\end{equation}
By Lemma \ref{lem:eta bdry}, we have $\eta_{c_i} ^0 (1, y_n) = \eta_{c_i} ^1 (1, y_n) = 1 - (1-c_i) (1- y_n)$. This shows (3).

\medskip

(4) We have $j= n+1$ with $\ell=0$ in \eqref{eqn:bidivision homotopy 0} so that 
\begin{equation}\label{eqn:bidivision homotopy 2}
(m_1, \ldots, m_n) = (y_1, \ldots, y_{i-1} , \eta_{c_i} ^0 (y_i, \epsilon), y_{i+1}, \ldots, y_n).
\end{equation}
When $\epsilon = 0$, $\eta_{c_i} ^0 (y_i, 0) = c_i y_i$, while when $\epsilon=1$, $\eta_{c_i} ^0 (y_i,1) = y_i$, by Lemma \ref{lem:eta bdry}. Thus \eqref{eqn:bidivision homotopy 2} is $\sigma_{c_i, i} (y_1, \ldots, y_n)$ if $\epsilon = 0$, and $(y_1, \ldots, y_n)$ if $\epsilon = 1$. This shows (4).

\medskip

(5) We have $j=n+1$ with $\ell=1$ in \eqref{eqn:bidivision homotopy 0} so that
\begin{equation}\label{eqn:bidivision homotopy 3}
(m_1, \ldots, m_n) = (y_1, \ldots, y_{i-1} , \eta_{c_i} ^1 (y_i, \epsilon), y_{i+1}, \ldots, y_n).
\end{equation}
When $\epsilon=0$, $\eta_{c_i} ^1 (y_i, 0) = 1- (1-c_i) y_i$, while when $\epsilon =1$, $\eta_{c_i} ^1 (y_i, 1) = 1$, by Lemma \ref{lem:eta bdry}. Thus \eqref{eqn:bidivision homotopy 3} is $\tau_i \circ \sigma_{1-c_i, i}$ if $\epsilon = 0$, while it is $(y_1, \ldots, y_{i-1}, 1, y_{i+1}, \ldots, y_n)$ if $\epsilon = 1$. This shows (5). 
\end{proof}

Using Lemmas \ref{lem:eta bdry} and \ref{lem:bidivision homotopy}, we now check that for a given normalized cycle, we have a concrete homotopy that relates the cycle to its bi-division at $y_i = c_i$. The arguments are straightforward, but the reader may need some patience to read the computational details.

\begin{prop}\label{prop:cdc}
Let $Y$ be a $k$-scheme of finite type.
Let $n \geq 1$ be an integer. Let $A_{\bullet}$ be one of the complexes
\begin{enumerate}
\item [1)] $z_d (Y, \bullet)$,
\item [2)] $  \mathcal{C}_{\bullet}= {\rm coker} (\rho^Y_U : z_d (Y, \bullet) \to z_d (U, \bullet))$ for a nonempty open $U \subset Y$.
\end{enumerate}

 Suppose $1 \leq i \leq n$. Let ${Z} \in A_n$ be a cycle such that $\partial_i ^{\epsilon} ({Z})$ is zero in $A_{n-1}$ for all $1 \leq i \leq n$ and $\epsilon \in \{ 0, 1 \}$. 

For a general $k$-rational point $c= (c_1, \ldots, c_n) \in \square_{\psi}^n \setminus \{ \mbox{faces} \}$ such that $(H_{(\ell), c, i} ^{n+1})^* ({Z}) \in A_{n+1}$ and $(H_{(\ell), c/j, i-1} ^n)^* (\partial_j ^{\epsilon} ({Z}))  \in A_n$ for all $1 \leq j \leq n$, $\epsilon \in \{ 0, 1 \}$ and $\ell \in \{ 0, 1 \}$ (which is possible by Lemma \ref{lem:Hell adm}), define
$$
\phi_{c,i} ({Z}):= (-1)^{n+1} \left( (H_{(0), c, i}^{n+1}) ^* ({Z}) - (H_{(1), c,i}^{n+1}) ^* ({Z})\right) \in A_{n+1}. 
$$

Then we have 
$$
\partial \phi_{c,i} ({Z}) = {Z} - \delta_{c_i, i} ({Z}) \mbox{ in } A_n.
$$
\end{prop}

\begin{proof}
%Strictly speaking, the Step 2 below, where we consider $n \geq 1$, includes the proof in Step 1 for $n=1$. But, the simpler case $n=1$ offers an illustration on how the argument goes, so let us keep it separately.

%\medskip

\textbf{Step 1:} Suppose $n=1$. 

Here $H_{(\ell), c, 1} ^2: \square_{\psi}^2 \to \square_{\psi}^1$ is given by $\eta_c ^{\ell}$ of Definition \ref{defn:eta} for $\ell=0,1$. We inspect the codimension $1$ faces.

\medskip

\emph{Case 1-1.} For $\iota_1 ^0 : \square_{\psi} ^1 \hookrightarrow \square_{\psi}^2$, $y_2 \mapsto (0, y_2)$, we have 
$$
H_{(\ell), c, 1} ^2 \circ \iota_1 ^0  (y_2) = \eta_c ^{\ell} (0, y_2) =\ell
$$
by Lemma \ref{lem:eta bdry}. 
The composite $H_{(\ell), c, 1} ^2 \circ \iota_1 ^0 $ is also equal to the composite
$$
\square_{\psi}^1 \overset{pr}{\to} \square_{\psi}^0 \overset{\iota_1 ^{\ell}}{ \hookrightarrow } \square_{\psi} ^1,
$$
which is the projection followed by the codimension $1$ face map. Hence $\partial_1 ^0 ( H_{(\ell), c, 1}^2)^* ({Z})$ is the pull-back of the face, $pr^* (\partial_1 ^{\ell} ({Z}))$. Since this is a degenerate cycle, we have $\partial_1 ^0 ( H_{(\ell), c, 1}^2)^* ({Z})=0$ in $A_1$.

\medskip

\emph{Case 1-2.} For $\iota_1 ^1: \square_{\psi} ^1 \hookrightarrow \square_{\psi} ^2$, $y_2 \mapsto (1, y_2)$, for $\ell = 0, 1$, we have
$$
H_{(\ell), c, 1} ^2 \circ \iota_1 ^1 (y_2)= \eta_c ^{\ell} (1, y_2) =  1 - (1-c) (1-y_2),
$$
by Lemma \ref{lem:eta bdry}, which is equal to 
$$
\tau_1 \circ \sigma_{1-c, 1} \circ \tau_1: \square_{\psi} ^1 \to \square_{\psi}^1.
$$
Thus for a general $c$, the cycles $\partial_1 ^1 ( H_{(\ell), c, 1} ^2)^* ({Z})$ are in $A_1$, and 
$$
\partial_1 ^1 (H_{(0), c, 1} ^2)^* ({Z}) = \partial_1 ^1 (H_{(1), c, 1} ^2)^* ({Z}).
$$

\medskip

\emph{Case 1-3.} For $\iota_2 ^0: \square_{\psi} ^1\hookrightarrow \square_{\psi} ^2$, $y_1 \mapsto (y_1, 0)$, we have
$$
H_{(\ell), c, 1} ^2 \circ \iota_2 ^0 (y_1) = \eta_c ^{\ell} (y_1, 0) = \tuborg c y_1, & \mbox{ if } \ell = 0, \\
1 - (1-c) y_1, & \mbox{ if } \ell =1,\sluttuborg
$$
by Lemma \ref{lem:eta bdry}. Here, the first is equal to $\sigma_{c,1}$ while the second is equal to $\tau_1 \circ \sigma_{1-c, 1}$. Thus the cycle $\partial_2 ^0 ( H_{(\ell), c, 1} ^2)^* ({Z})$ is $\sigma_{c,1} ^* ({Z})$ for $\ell =0$ and $\sigma_{1-c, 1} ^* \tau_1 ^* ({Z})$ for $\ell = 1$.

\medskip

\emph{Case 1-4.} For $\iota_2 ^1: \square_{\psi} ^1 \hookrightarrow \square_{\psi}^2$, $y_1 \mapsto (y_1, 1)$, we have
$$
H_{(\ell), c,1} ^2 \circ \iota_2 ^1 (y_1) = \eta_c ^{\ell} (y_1, 1) = \tuborg y_1, & \mbox{ if } \ell = 0, \\
1, & \mbox{ if } \ell = 1,\sluttuborg
$$
by Lemma \ref{lem:eta bdry}.

When $\ell=0$, we have $\partial_2 ^1 ( H_{(0), c, 1} ^2)^* ({Z}) = {Z}$, which is in $A_1$ already. 

When $\ell = 1$, the morphism is equal to the composite
$$
\square_{\psi} ^1 \overset{pr}{\to} \square_{\psi} ^0 \overset{\iota_1 ^1}{\hookrightarrow} \square_{\psi} ^1,
$$
which is the projection followed by the codimension $1$ face map. Hence $\partial_2 ^1 ( H_{(1), c, 1} ^2)^* ({Z})$ is the pull-back of the face, $pr^* (\partial_1 ^1 ({Z}))$. Since this is a degenerate cycle, we have $\partial_2 ^1 ( H_{(1), c, 1} ^2)^* ({Z}) = 0$ in $A_1$.

\medskip

The above codimension $1$ face calculations also show that (only Case 1-3 and Case 1-4 have the non-cancelling surviving terms)
$$
\partial \phi_{c, 1} ({Z}) = {Z} - (\sigma_{c,1} ^* ({Z}) - \sigma_{1-c, 1}^* \tau_1^* ({Z})) = {Z} - \delta_{c,1} ({Z}).$$
This proves the proposition for $n=1$.

\bigskip

\textbf{Step 2:} Now let $n \geq 1$. 

 For the following Case 2-1 $\sim$ Case 2-3, let $1 \leq i, j \leq n$ and $\epsilon \in \{0, 1 \}$. The Cases 2-4, 2-5 deal with $j=n+1$.

\medskip

\emph{Case 2-1.} Suppose $j \not = i$. The identity $H_{(\ell), c, i} ^{n+1} \circ \iota_j ^{\epsilon} = \iota_j ^{\epsilon} \circ H_{(\ell), c/j, i-1}^{n}$ for $j<i$ in Lemma \ref{lem:bidivision homotopy}-(1) reads 
\begin{equation}\label{eqn:cdc00}
\partial_j ^{\epsilon} (H _{(\ell), c, i} ^{n+1})^* ({Z}) = (H_{(\ell), c/j, i-1} ^n)^* (\partial_j ^{\epsilon} ({Z})).
\end{equation}

We are given that $\partial_j ^{\epsilon} ({Z})$ is trivial in $A_{n-1}$ for all $1 \leq j \leq n$ and $\epsilon \in \{ 0, 1 \}$. Thus by \eqref{eqn:cdc00}, we have $\partial_j ^{\epsilon} ( H_{(\ell), c, i} ^{n+1})^* ({Z})$ is trivial as well. 

Similarly for $j>i$, using the second part of Lemma \ref{lem:bidivision homotopy}-(1), we deduce that $\partial_j ^{\epsilon} (H _{(\ell), c, i} ^{n+1})^* ({Z}) = (H_{(\ell), c/j, i-1} ^n)^* (\partial_j ^{\epsilon} ({Z}))$ is trivial from that $\partial_j ^{\epsilon} (Z)$ is trivial in $A_{n-1}$.

\medskip

\emph{Case 2-2.} Suppose $j=i$ and $\epsilon = 0$. By Lemma \ref{lem:bidivision homotopy}-(2), we know that $H_{(\ell), c, i} ^{n+1} \circ \iota_i ^{0}$ is equal to
$$\square_{\psi}^{n} \to \square_{\psi}^{n}, \ \ (y_1, \ldots, y_{n})\mapsto (y_1, \ldots, y_{i-1}, \ell, y_i, \ldots, y_{n-1}).$$
 It can also be written as the composite
$$ 
\square_{\psi}^n \overset{pr_n}{\to} \square_{\psi} ^{n-1} \overset{\iota_i ^{\ell}}{\hookrightarrow } \square_{\psi} ^n
$$
of the projection $pr_n$ that ignores $y_n$ and the face map $\iota_i ^{\ell}$. Hence $\partial_i ^0 (H_{(\ell), c, i}^{n+1})^* ({Z})$ is equal to the pull-back by the projection, $pr_n ^* ( \partial_i ^{\ell} ({Z})) = (\partial_i ^{\ell} ({Z})) \times \square_{\psi}^1$. Since this is a degenerate cycle, we have $\partial_i ^0 (H_{(\ell), c, i}^{n+1})^* ({Z}) = 0$ in $A_{n}$. 
In particular, $\partial_i ^{0} \phi_{c, i} ({Z}) = 0$. 

\medskip

\emph{Case 2-3.} Suppose $j=i$ and $\epsilon = 1$. By \eqref{eqn:ulala1} in the proof of Lemma \ref{lem:bidivision homotopy}-(3), we have
$$
H_{(\ell), c, i} ^{n+1} \circ \iota_i = (m_1, \ldots, m_n) = (y_1, \ldots, y_{i-1}, \eta_{c_i} ^{\ell} (1, y_n), y_i, \ldots, y_{n-1}),
$$
with $\eta_{c_i} ^{\ell} (1, y_n) = 1- (1-c_i) (1- y_n)$. This is equal to the composite
$$
\square_{\psi} ^n \overset{\tau_n}{\to} \square_{\psi}^n \overset{\tau}{\to} \square_{\psi}^n \overset{ \sigma_{1- c_i, i}}{\to} \square_{\psi}^n \overset{\tau_i}{\to} \square_{\psi}^n,$$
where the map $\tau$ is the permutation $(y_1, \ldots, y_n)\mapsto (y_1, \ldots, y_{i-1}, y_n, y_i, \ldots, y_{n-1})$. 
Furthermore, by Lemma \ref{lem:bidivision homotopy}-(3), we have $\partial_i ^1  (H_{(0), c, i}^{n+1})^* ({Z}) =\partial_i ^1  (H_{(1), c, i}^{n+1})^* ({Z}) $, so that $\partial_i ^1 \phi_{c, i} ({Z}) = 0$. 

\medskip

\emph{Cases 2-4, 2-5.} It remains to inspect $\partial_{n+1} ^{\epsilon} ( H_{(\ell), c, i} ^{n+1})^* ({Z})$ and $\partial_{n+1} ^{\epsilon} \phi_{c,i} ({Z})$ for $\epsilon=0, 1$.

The calculations of Lemma \ref{lem:bidivision homotopy}-(4),(5) show that 
\begin{equation}\label{eqn:case145star}
 \tuborg \partial _{n+1} ^0 (H_{(0), c, i} ^{n+1} )^* ({Z}) = \sigma_{c_i, i} ^* ({Z}), \ \ \  \partial _{n+1} ^1 (H_{(0), c, i} ^{n+1} )^* ({Z})= {Z},\\
\partial _{n+1} ^0 (H_{(1), c, i} ^{n+1} )^* ({Z})= \sigma_{1-c_i, i} ^* \tau_i ^* ({Z}),  \ \ \ \partial _{n+1} ^1 (H_{(1), c, i} ^{n+1} )^* ({Z})=0.\sluttuborg
\end{equation}
Here, the last equality $\partial _{n+1} ^1 (H_{(1), c, i} ^{n+1} )^* ({Z})=0$ comes from that $\partial _{n+1} ^1 (H_{(1), c, i} ^{n+1} )^* ({Z})$ is a degenerate cycle; this is because, by Lemma \ref{lem:bidivision homotopy}-(5), the composite $H_{(1), c, i} ^{n+1} \circ \iota_{n+1} ^1$ can be rewritten as the composite
$$
\square_{\psi} ^n \overset{pr_i}{\to} \square_{\psi} ^{n-1} \overset{\iota_i ^1}{\hookrightarrow} \square_{\psi} ^n
$$
of the projection $pr_i$ that drops $y_i$ and the face map $\iota_i ^1$, so the cycle is a permutation of the degenerate cycle $\partial_i ^1 ({Z}) \times \square_{\psi}^1$, thus also degenerate.

\medskip

Combining the above codimension $1$ face calculations, we have:
\begin{eqnarray*}
& & \partial \phi_{c, i} ({Z})= (-1)^{n+1} ( \partial _{n+1} ^1 - \partial _{n+1} ^0)  \phi_{c,i} ({Z})\\
& = & (-1)^{n+1} ( \partial _{n+1} ^1 - \partial _{n+1} ^0) \cdot (-1)^{n+1} ( (H_{(0), c, i} ^{n+1})^* ({Z} )- ( H_{(1), c, i} ^{n+1})^* ({Z}))\\
&=& {Z} - (\sigma_{c,i} ^* ({Z}) - \sigma_{1-c_i, i}^* \tau_ i ^* ({Z}) ) = {Z} - \delta_{c_i, i} ({Z}).
\end{eqnarray*}
This proves the proposition for $n \geq 2$.
\end{proof}

We need the following to apply the bi-division operations successively: 

\begin{lem}\label{lem:0 face}
Let $Y$ be a $k$-scheme of finite type.
Let $n \geq 1$ be an integer. Let $A_{\bullet}$ be one of the complexes
\begin{enumerate}
\item [1)] $z_d (Y, \bullet)$,
\item [2)] $  \mathcal{C}_{\bullet}= {\rm coker} (\rho^Y_U : z_d (Y, \bullet) \to z_d (U, \bullet))$ for a nonempty open $U \subset Y$.
\end{enumerate}

 Let ${Z} \in A_n$ be a cycle such that $\partial_i ^{\epsilon} ({Z})$ is trivial in $A_{n-1}$ for all $1 \leq i \leq n$ and $\epsilon \in \{ 0, 1 \}$. 

 Then for any general $k$-rational point $c= (c_1, \ldots, c_n) \in \square^n_{\psi} \setminus \{ \mbox{\rm faces}\}$, the same property holds for $\delta_{c_i,i} ({Z})$.
\end{lem}

\begin{proof}
We show that all codimension $1$ faces of $\delta_{c_i,i} ({Z}) = \sigma_{c_i, i} ^* {Z} -  \sigma_{1-c_i, i} ^* \tau_i ^* {Z}$ are trivial. We again use composites of morphisms. 

\medskip

When $j \not = i$, first consider the case $i<j$. Then we have
$$
\sigma_{c_i, i} \circ \iota_j ^\epsilon = \iota_j ^\epsilon \sigma_{c_i, i}, \ \ \tau_i \circ \sigma_{1-c_i, i} \circ \iota_j ^{\epsilon} = \iota_j ^{\epsilon} \circ \tau_i \circ \sigma_{1-c_i, i}. 
$$
Hence $\partial_j ^{\epsilon} \delta_{c_i, i} ({Z}) = \delta_{c_i, i} (\partial_j ^{\epsilon} ({Z})) = \delta_{c_i,i} (0) = 0$. When $i>j$, we have
$$
\sigma_{c_i, i} \circ \iota_j ^{\epsilon} = \iota_j ^{\epsilon} \circ \sigma_{c_i, i-1}, \ \ \tau_i \circ \sigma_{1-c_i, i} \circ \iota_j ^{\epsilon} = \iota_j ^{\epsilon} \circ \tau_{i-1} \circ \sigma_{1- c_i, i-1}$$
so that we deduce that
\begin{eqnarray*}
\partial_j ^{\epsilon} \delta_{c_i, i} (Z) & =& \partial_j ^{\epsilon} ( \sigma_{c_i, i} ^* (Z) - \sigma_{1-c_i, i} ^* \tau_i ^* (Z))= (\sigma_{c_i, i-1} ^*- \sigma_{1-c_i, i-1} ^* \tau_{i-1}^* ) (\partial_j ^{\epsilon} (Z))\\
&= & (\sigma_{c_i, i-1} ^*- \sigma_{1-c_i, i-1} ^* \tau_{i-1}^* )  (0) = 0.
\end{eqnarray*}

\medskip

When $j=i$ and $\epsilon = 0$, we have
$$
\sigma_{c_i, i} \circ \iota_i ^0 = \iota_i ^0, \ \ \tau_i \circ \sigma_{1- c_i, i} \circ \iota_i ^0 = \iota_i ^1.
$$
Hence $\partial_i ^0 \delta_{c_i,i} ({Z}) = \partial_i ^0 ({Z}) - \partial_i ^1 ({Z}) = 0$.

\medskip

When $j=i$ and $\epsilon=1$, we have
$$
\sigma_{c_i, i} \circ \iota_i ^1 = \tau_i \circ \sigma_{1-c_i, i} \circ \iota_i ^1.
$$
Hence $\partial_i ^1 \sigma_{c_i, i} ^* {Z} = \partial_i ^1 \sigma_{1-c_i, i} ^* \tau_i ^* {Z}$, i.e. $\partial_i ^1 \delta_{c_i,i} ({Z}) = 0$. 

Thus we have checked that all codimension $1$ faces of $\delta_{c_i, i} ({Z})$ are trivial.
\end{proof}

\subsubsection{Cubical subdivision}\label{sec:cube subdiv 0}
By iterating the bi-division operations of \S \ref{sec:bi-division} and \S \ref{sec:homo bidivision}, we can define the cubical subdivisions:

\begin{defn}\label{defn:cub sd}
Let $Y$ be a $k$-scheme of finite type.
Let ${Z} \in z_d ({Y}, n)$. For a general $k$-rational point $c= (c_1, \ldots, c_n) \in \square_{\psi} ^n \setminus \{ \mbox{faces} \}$, define
\begin{equation}\label{eqn:cub sd bidiv}
{\rm sd}_c ({Z}):= \delta_{c_n, n} \cdots \delta_{c_1, 1} ({Z}),
\end{equation}
which we call the \emph{cubical subdivision} ${\rm sd}_c ({Z})$ with respect to $c$.
\qed
\end{defn}

We note that:

\begin{lem}\label{lem:M=0 adm}
Let $Y$ be a $k$-scheme of finite type. Let ${Z}\in z_d ({Y}, n)$. 
 
 Then for a general $k$-rational point $c \in \square_{\psi} ^n \setminus \{ \mbox{\rm faces} \}$, we have ${\rm sd}_c ({Z}) \in z_d ({Y}, n)$. 
 \end{lem}

\begin{proof}
Given Definitions \ref{defn:bidivision delta} and \ref{defn:cub sd}, this follows immediately by repeatedly applying Lemma \ref{lem:scaling adm}, with the help of Lemma \ref{lem:0 face}.
\end{proof}

\begin{remk}\label{remk:cub sd bidiv 2}
The cubical subdivision ${\rm sd}_c (Z)$ in \eqref{eqn:cub sd bidiv} also admits a description of the form
\begin{equation}\label{eqn:cub sd bidiv 2}
{\rm sd}_c ({Z})= \sum_{v}  \varepsilon (v) \iota_v ^* \pi_v ^* ({Z})
\end{equation}
for the sign $\varepsilon (v)$ of \eqref{eqn:sgn_orientation}, where the sum is taken over all vertices of $\square_{\psi}^n$. Explicitly, $\pi_v$ and $\iota_v$ are defined as follows. 

Let $n \geq 1$ and let $v \in \square^n _{\psi} $ be a vertex. Let $(y_1, \ldots, y_n )\in \square_{\psi} ^n$ be the coordinates. Define the affine isomorphism $\pi_v: \square_{\psi} ^n \to \square^n_{\psi}$
\begin{equation}\label{eqn:pi 1}
y_i \mapsto \tuborg y_i,  &\mbox{ if } v_i = 0, \\ 1- y_i, & \mbox{ if } v_i = 1.\sluttuborg
\end{equation}

For a $k$-rational point $c= (c_1, \ldots, c_n) \in \square^n _{\psi}\setminus \{ \mbox{faces} \}$, define $\iota_v: \square_{\psi} ^n \to \square_{\psi}^n$ by
\begin{equation}\label{eqn:iota 1}
y_i \mapsto \tuborg c_i y_i, & \mbox{ if } v_i = 0, \\
(1-c_i)y_i, & \mbox{ if } v_i = 1.\sluttuborg
\end{equation}
Here, the map $\iota_v$ depends on the choice of $c$.

So, the composite $\pi_v \circ \iota_v: \square_{\psi} ^n \to \square_{\psi} ^n$ is explicitly
\begin{equation}\label{eqn:pi iota}
y_i \mapsto \tuborg c_i y_i, & \mbox{ if } v_i= 0,\\ 
1- (1-c_i)y_i, & \mbox{ if } v_i = 1.\sluttuborg
\end{equation}
Note that $\pi_v$ is the composite of the involutions $\tau_j$ for $ j \in B(v)$, and $\iota_v$ is a composite of the scaling morphisms in Definition \ref{defn:involution}. (Recall that when $v= (\epsilon_1, \ldots, \epsilon_n)$ with $\epsilon_j \in \{ 0, 1 \}$, we defined $B(v) := \{ j \ | \ \epsilon_j = 1 \}$.)
\qed
\end{remk}

Let $n \geq 1$. Using the normalization theorems (Theorems \ref{thm:normalization0} and \ref{thm:normalization}), we may choose a representative cycle of a given class in $\CH^q ({Y}, n)$ (resp. ${\rm H}_n (\mathcal{C}_{\bullet})$), whose codimension $1$ faces $\partial_i ^{\epsilon}$ are all trivial in $z_d ({Y}, n-1)$ (resp. $\mathcal{C}_{n-1}$). For such normalized representative cycles, we can further replace them by their cubical subdivisions.

\begin{prop}\label{prop:cub sd}
Let $Y$ be a $k$-scheme of finite type.
Let $A_{\bullet}$ be one of the bounded above complexes
\begin{enumerate}
\item [1)] $z_d (Y, \bullet)$,
\item [2)] $  \mathcal{C}_{\bullet}= {\rm coker} (\rho^Y_U : z_d (Y, \bullet) \to z_d (U, \bullet))$ for a nonempty open $U \subset Y$.
\end{enumerate}

 Let $n \geq 1$. Let ${Z} \in A_n$ be a cycle such that $\partial_i ^{\epsilon} ({Z}) $ is trivial in $A_{n-1}$ for all $1 \leq i \leq n$ and $\epsilon = 0, 1$.

Then for a general $k$-rational point $c \in \square_{\psi}^n \setminus \{ \text{faces} \}$ such that ${\rm sd}_c ({Z}) \in A_n$, there exists a cycle $\phi_n ({Z}) \in A_{n+1}$ satisfying
$$
\partial \phi_n ({Z}) = {Z} - {\rm sd}_c ({Z}).
$$
Furthermore, $\partial_i ^{\epsilon} ({\rm sd}_c ({Z})) $ is trivial in $A_{n-1}$ for all $1 \leq i \leq n$ and $\epsilon = 0,1$, again.
\end{prop}

\begin{proof}
For a general $k$-rational point $c= (c_1,\ldots, c_n) \in \square_{\psi}^n \setminus \{ \mbox{faces} \}$, we have the equivalence ${Z} \equiv \delta_{c_1,1} ({Z})$ modulo the boundary of $ \phi_{c, 1}({Z})$ by Proposition \ref{prop:cdc}. The cycle $\delta_{c_1,1}({Z})$ has the trivial codimension $1$ faces in $A_{n-1}$ again by Lemma \ref{lem:0 face}. Inductively we may apply the bi-division process to get a sequence of equivalences of cycles modulo boundaries
$$
Z \equiv \delta_{c_1,1} ({Z}) \equiv \delta_{c_2,2} \delta_{c_1,1} ({Z}) \equiv \cdots \equiv \delta_{c_n,n} \cdots \delta_{c_1,1} ({Z}).
$$
By Definition \ref{defn:cub sd}, ${\rm sd}_c ({Z})= \delta_{c_n,n} \cdots \delta_{c_1,1} ({Z})$. This shows that ${Z}$ is equivalent to ${\rm sd}_c ({Z})$, modulo the boundary of
$$
\phi_n ({Z}) := \phi_{c, 1} ({Z}) + \phi_{c, 2} (\delta_{c_1, 1} ({Z})) + \cdots + \phi_{c,n} ( \delta_{c_{n-1}, n-1} \cdots \delta_{c_1, 1} ({Z})).
$$
This proves the proposition.
\end{proof}

\subsection{Blow-up of faces of a cube}\label{sec:moving by blow-up}
In \S \ref{sec:moving by blow-up}, we recall a few terminologies and results related to the techniques of ``blowing-up of faces of a cube" from S. Bloch \cite{Bloch moving} and M. Levine \cite{Levine JPAA}, \cite{Levine moving}. 

There is nothing original in \S \ref{sec:moving by blow-up}, and this is included to explain in a self-contained manner the notations and terminologies used in the essential result, Theorem \ref{thm:bloch-levine}, taken from \cite{Bloch moving}, \cite{Levine JPAA}, \cite{Levine moving}. 

We continue to suppose that $k$ is an infinite field until the end of \S \ref{sec:moving by blow-up}. 

\subsubsection{Blow-up of faces and coordinate neighborhoods}
The main references are S. Bloch \cite{Bloch moving} and M. Levine \cite{Levine moving}. We frequently visit \cite[\S 2 $\sim$ \S 4]{Levine moving}.

\begin{defn}[{\cite[(1.1)]{Bloch moving}}, {\cite[\S 3.1]{Levine moving}}]\label{defn:face blow}
Let $S$ be a smooth quasi-projective $k$-scheme. Let $D_1, \ldots, D_N$ be a set of smooth effective divisors on $S$ such that for each subset $I \subset \{ 1, \ldots, N\}$ the intersection $D_I:=  \bigcap_{i \in I} D_i$ is either transverse or empty. We write $\partial S$ for both the set of the divisors as well as the union of them, when no confusion arises.

For a subset $I \subset \{ 1 , \ldots, N \}$, let $\pi: T \to S$ be the blow-up of $S$ along $D_I$. For $1 \leq i \leq N$, let $D_i'$ be the strict transform of $D_i$ via $\pi$, while let $D_{N+1} '$ be the exceptional divisor of the blow-up. The set of the divisors $\{ D_1', \ldots, D_N', D_{N+1}' \}$ (or their union) is written as $\partial T$. The divisors in $\partial T$ enjoy the same transverse intersection property. We call them the \emph{distinguished divisors} on $T$.

We can repeat this process finitely many times, where each blow-up has the center given by the intersection of a subset of the distinguished divisors. We then obtain a scheme $U \to S$ with a set of smooth Cartier divisors $\delta_1, \ldots, \delta_N$, with the same transverse intersection property; we call them the \emph{distinguished divisors} on $U$, as well.

Intersections of distinguished divisors are called \emph{faces}, and a face of dimension $0$ is called a \emph{vertex}. A face of dimension $1$ is called an \emph{edge}.

Let $\mathcal{B}_S$ be the full subcategory of the category of schemes over $S$, whose objects are the varieties $U \to S$ with the distinguished divisors obtained as the above. 
 \qed
\end{defn}

\begin{defn}[{cf. \cite[(1.1)]{Bloch moving}}] 

Consider $(S, \partial S = \{ D_1, \ldots, D_N\})$ as before. Let $X$ be a $k$-scheme of finite type, and let $f: {X} \to S$ be a morphism of $k$-schemes. Suppose that $f({X}) \not\subset \bigcup_i D_i$. We say that $f$ (or ${X}$) \emph{intersects the faces properly} if for any face $\sigma \subset S$ of codimension $r$, $f^{-1} (\sigma)=   {X}  \times_S \sigma $ has codimension $\geq r$ in ${X}$. For any object $\pi: T \to S$ in $\mathcal{B}_S$, the \emph{strict transform} $f' : {X}' \to T$ of $f$ is defined to be the closure in ${X} \times _S T$ of $f^{-1} ( S \setminus \bigcup_i D_i)$. It is often denoted by $\pi^! {X}$. We can naturally extend it to formal finite sums of such morphisms ${X} \to S$. \qed
\end{defn}

\begin{defn}[{\cite[p.316, l.7--13]{Levine moving}}]
Let $(S, \partial S)$ be as before. Let $v \in S$ be a vertex. Let $D_{i_1}, \ldots, D_{i_n}$ be the divisors in $\partial S$ containing $v$. 

An $n$-tuple of regular functions $(f_1, \ldots, f_n)$ defined in a neighborhood $U$ of $v$ is called \emph{a coordinate system at $v$} (adapted to $\partial S$), if the map
$$
(f_1, \ldots, f_n): U \to \mathbb{A}^n
$$
is an open immersion, and each $D_{i_j} \cap U$ is given by $f_j = 0$ for $1 \leq j \leq n$.\qed
\end{defn}

\begin{defn}[{\cite[p.316, l.14]{Levine moving}}, cf. {\cite[(1.3)]{Bloch moving}}]\label{defn:coordinate}
Let $(S, \partial S)$ be as before. Suppose each vertex $w \in S$ has an open neighborhood $U_w \subset S$ with a coordinate system $f^w= (f_1 ^w, \ldots, f_n ^w)$ at $w$ on $U_w$. Call this coordinate system or any of its reordering, a \emph{distinguished coordinate system at $w$}. 

\medskip

Let $F \subset S$ be a face. 
Let $p: S' \to S$ be the blow-up of $S$ with the center $F$, and let $E \subset S'$ be its exceptional divisor. We obtain the distinguished divisors $\partial S'$ as in Definition \ref{defn:face blow}. 

Let $v \in S'$ be a vertex. Let $w:= p (v)$. We define a distinguished coordinate system at $v$ from that at $w$ as follows. 

If $w \not \in F$, then we take $U_v:= p^{-1} (U_w \setminus F)$ and define $f^v := p^* (f^w) = f^w \circ p$. 

If $w \in F$, then since $p$ is the blow-up of the face $F$, there exist a subset $J \subset \{ 1, \ldots, n \}$ and an index $i \in J$ such that $F \cap U_w$ is defined by $f_j ^w= 0$ for $j \in J$, and
$$
v = E \cap \bigcap _{ j \in J, j \not = i} [\{ f_j ^w = 0 \}],
$$
where $[ \{- \} ]$ denotes the strict transform of the locus $\{ - \}$. Let $U_v:= p^{-1} (U_w) \setminus  [ \{ f_i ^w = 0 \}]$, and define for $1 \leq j \leq n$
$$
f_j ^v := \tuborg p^* (f_j ^w) , & j \not \in J \setminus \{ i \}, \\
p^* ( f_j ^w / f_i ^w), & j \in J \setminus \{ i \}.\sluttuborg
$$
Here, we allow reordering of $f_j ^w$ over $j$.\qed
\end{defn}

\begin{exm}[{\cite[Example 3.3]{Levine moving}}] 
We are primarily interested in the case when $S= S_0:= (\mathbb{A}^1)^n$, with the $2n$ distinguished divisors given by $y_i = 0, 1$ for $1 \leq i \leq n$, where  $(y_1, \ldots, y_n) \in S_0$. We order them by 
$$
D_i := \tuborg
 \{ y_i = 0 \}, & \mbox{ for } i = 2k, 1 \leq k \leq n, \\
\{ y_i = 1 \}, & \mbox{ for } i = 2k-1, 1 \leq k \leq n.
\sluttuborg
$$
Each vertex $v$ of $S_0$ is of the form $ (\epsilon_1, \ldots, \epsilon_n)$ for some $\epsilon_i \in \{ 0, 1 \}$. We choose the following distinguished coordinate system $y^v := (y_1 ^v, \ldots, y_n ^v)$ near $v$ by
\begin{equation}\label{eqn:dist coord}
y_i ^v := \tuborg 
y_i, & \mbox{ if } \epsilon_i = 0, \\
1- y_i, & \mbox{ if } \epsilon _i = 1.
\sluttuborg
\end{equation}

When $S_M \to \cdots \to S_1 \to S_0$ is a tower of blow-ups of faces as in Definition \ref{defn:face blow}, for each vertex $v$ of each $S_j$, we define a distinguished coordinate system at $v$ as in Definition \ref{defn:coordinate} inductively from the above distinguished coordinates near the vertices of $S_0$.
\qed
\end{exm}

Fix an order on $\partial S_M= \{ D_1, \ldots, D_N\}$. Let $v \in S_M$ be a vertex and let $y^v= ( y_1 ^v, \ldots, y_n ^v)$ be a given coordinate system. The closure of the divisor $\{ y_i ^v = 0 \}$ is one of $\partial S_M$, call it $D_{j (i)}$. Following \cite[\S 4.3]{Levine moving}, we say that $y^v$ is ordered if $j(1) < j(2) < \cdots < j(n)$. This requirement determines an order on the coordinates $y^v$ uniquely by a chosen order on $\partial S_M$. We use this ordering, when we need one.

\begin{defn}[{\cite[Definition (1.3.3)]{Bloch moving}} or {\cite[Definition 4.4]{Levine moving}}]\label{defn:edge coordinate}

Every edge $\ell$ in $S_M$ contains exactly two vertices (see \cite[Lemma (1.3.2)]{Bloch moving} or \cite[Lemma 3.2]{Levine moving}), call them $v$ and $w$. We say $v$ and $w$ are adjacent vertices connected by the edge $\ell$.

Let $(t_1, \ldots, t_n)$ and $(u_1, \ldots, u_n)$ be coordinate systems around $v$ and $w$, respectively. The \emph{coordinate for $\ell$ at $v$} is the unique $t_p$ such that $t_p$ does not vanish on $\ell$. Similarly, let $u_q$ be the coordinate for $\ell$ at $w$. We define $g= g(v, w) \in \mathfrak{S}_n$ to be the unique permutation of $\{1, \ldots, n \}$ such that $g (p) = q$, and for $i \not = p$, the distinguished divisors defined locally by $\{t_i = 0\}$ and $\{u_{g(i)}=0\}$ coincide.\qed
\end{defn}

We have the sign $\varepsilon (v)$, called the orientation, for each vertex $v \in S_M$. See \cite[Definition 4.4]{Levine moving}. We do not recall its full definition, but we mention the following: firstly, when $v, w$ are adjacent vertices of $S_M$ connected by an edge $\ell$, then $\varepsilon (v) = - \sgn (g (v,w)) \varepsilon (w)$. See \cite[Lemma 4.5]{Levine moving} (cf. \cite[Lemma (3.1.2)]{Bloch moving}). Secondly, when $M=0$, this sign $\varepsilon (v)$ is identical to the sign \eqref{eqn:sgn_orientation} in \S \ref{sec:bi-division}. We use this in Lemma \ref{lem:sd^0 comparison} later.

\subsubsection{Complexes associated to schemes and subschemes}\label{sec:complex scheme}
In \S \ref{sec:complex scheme}, we recall from M. Levine \cite[\S 2.1, \S 2.2]{Levine moving}, the notion of complexes of schemes used for the rest of \S \ref{sec:moving by blow-up}. 

\medskip

For an additive category $\mathcal{A}$, recall that the DG-category $\mathbf{C} (\mathcal{A})$ has for its objects the complexes $A=A^{\bullet}$ of objects in $\mathcal{A}$, while for the Hom-complex between two objects $A$ and $B$, we let
$$
\Hom ((A, d_A), (B, d_B)):= (\Hom^{\bullet} ((A, d_A), (B, d_B)), d),
$$
where 
$$
\Hom^n ((A, d_A), (B, d_B)) = \{ f= (f^i: A^i \to B^{i+n}, i \in \mathbb{Z}) \},
$$
and for $f= (f^i) \in \Hom^n ((A, d_A), (B, d_B))$, we have
$$
df:= ( d_B ^{i+n} \circ f^i + (-1)^{n-1} f^{i+1} \circ d_A ^i: A^i \to B^{i+n+1}, i \in \mathbb{Z}).
$$
A map $f \in \Hom^n$ is called a morphism of complexes if $df =0$. 

\medskip

Here is one example of $\mathcal{A}$ from \cite[\S 2.2]{Levine moving}. 
Consider the additive category $\mathcal{A}=\mathbb{Z} \Sch_k$. Its objects are disjoint unions of connected $k$-schemes, and $\Hom_{\mathcal{A}} (X, Y)$ is the free abelian group on the $k$-morphisms between connected components of $X$ and $Y$, respectively. The empty scheme is $0$. We form the DG-category $\mathbf{C} (\mathbb{Z} \Sch_k )$ as the above. Instead of the above $\Sch_k$, we may also use $\Sch_k ^{\ess}$ obtained from $\Sch_k$ by adding localizations of the objects of $\Sch_k$ at finite subsets as objects.

\begin{exm}[\emph{ibid.}]\label{exm:Levine cat2}
Let $X \in \Sch_k$. Let $D_i \hookrightarrow X$ be closed immersions for $1 \leq i \leq N$. The complex $(X; D_1, \ldots, D_N)\in \mathbf{C}(\mathbb{Z}\Sch_k)$ is defined as follows. For each subset $I \subset \{ 1, \ldots, N \}$, let $D_I := \bigcap_{i \in I} D_i$, where for $I= \emptyset$, let $D_I:= X$.

First consider the $N$-dimensional complex, whose object at the multi-degree $J= (-j_1, \ldots, -j_N) \in \mathbb{Z}^N$ is 
$$
\tuborg D _{I (J)}, & \mbox{ if } \mbox{ all of $j_i$ is either $0$ or $1$, and } I(J) = \{ i \ | \ j_i = 1 \},\\
0, & \mbox{ otherwise,}\sluttuborg
$$
and all maps are the morphisms given by the induced closed immersions. For $j \in I$, the relevant morphism is denoted by $\iota_{I, j}: D_I \to D_{I\setminus \{ j \}}$. 

The complex $(X; D_1, \ldots, D_N)$ is the total complex of the above $N$-dimensional complex. For $I= (i_1< \cdots < i_r)$ and $i_j \in I$, the map for the complex is given by taking sums of $(-1)^{(j-1)} \iota_{I, i_j}: D_I \to D_{I \setminus \{i_j \}}.$

\medskip

For instance, when $N=2$, the $2$-dimensional complex appears as the left square below, with their respective bi-degrees written on the right:
$$
\xymatrix{
D_1 \ar[r] & X  & & & (-1, 0) & (0,0) \\
D_1 \cap D_2 \ar[u] \ar[r] & D_2, \ar[u] & & & (-1, -1) & (0, -1).}
$$
Hence $(X; D_1, D_2)$ is given by
$$
D_1 \cap D_2 \overset{ ( - \iota_{I, 2}, \iota_{I, 1} )}{ \longrightarrow} D_1 \coprod D_2 \longrightarrow X,
$$
where the objects are in the degrees $-2, -1, 0$, respectively.\qed
\end{exm}

\subsubsection{The cubical complex}

We recall some discussions from M. Levine \cite[\S 4.6]{Levine moving}. For $1 \leq j \leq n+1$ and $\epsilon \in \{ 0, 1 \}$, we have the codimension $1$ faces $\iota_{j} ^{\epsilon} : \square_{\psi} ^n \hookrightarrow \square_{\psi} ^{n+1}$. Let $(\square_{\psi}^n)_0$ denote the localization of $\square_{\psi} ^n$ at the $2n$ vertices. We similarly have $\iota_{j} ^{\epsilon} : (\square_{\psi} ^n)_0 \hookrightarrow ( \square_{\psi} ^{n+1})_0$. Define the formal sums of morphisms
$$
d_+ ^n, d_- ^n: \square_{\psi} ^n \to \square_{\psi} ^{n+1},
$$
$$
d_+ ^n:= \sum_{j=1} ^{n+1} (-1)^{j+1} \iota_j ^0, \ \ \ \ d_- ^n:= \sum_{j=1} ^{n+1} (-1)^{j} \iota_j ^1.
$$
They induce the following objects in $\mathbf{C} (\mathbb{Z}\Sch_k)$
$$
(\square_{\psi} ^{\bullet}, d_+), \ \ (\square_{\psi} ^{\bullet}, d_-), \ \ (\square_{\psi} ^{\bullet}, d= d_+ + d_-),
$$
and
$$
((\square_{\psi} ^{\bullet})_0, d_+),  \ \ ((\square_{\psi} ^{\bullet})_0, d_-),  \ \ ((\square_{\psi} ^{\bullet})_0, d),
$$
as well as morphisms of complexes $( (\square_{\psi} ^{\bullet})_0, d_? ) \to (\square_{\psi} ^{\bullet}, d_?)$ for $?= +, -, \emptyset$.

Here, for instance $(\square_{\psi} ^{\bullet}, d_+)$ means the complex
$$
\square_{\psi} ^0 \overset{d_+}{\to} \square_{\psi} ^1 \overset{d_+}{\to} \square_{\psi} ^2 \cdots \overset{d_+}{\to}  \square_{\psi}^3 \overset{d_+}{\to}  \cdots ,
$$
with $\square_{\psi}^n$ in degree $n$. The identity $d_+ \circ d_+ = 0$ holds by the cubical formalism.

\subsubsection{Cubiculations}
We recall cubiculations following M. Levine \cite[\S 4.7]{Levine moving} (or the paragraphs after S. Bloch \cite[Definition 1.3.3]{Bloch moving}). 

Let $c = (c_1, \ldots, c_n) \in S_0 \setminus \{ \mbox{the $2n$ distinguished divisors} \}$ be a $k$-rational point. Let $\pi: S_M \to \cdots \to S_1 \to S_0 = \mathbb{A}^n$ be a tower of blow-ups of faces as in Definition \ref{defn:face blow}. For $0 \leq j \leq M$, let $U_j := S_j \setminus \partial S_j$. Then $\pi$ gives the natural identifications
\begin{equation}\label{eqn:cvn}
\pi^{-1}: S_0 \setminus \{ \mbox{the $2n$ distinguished divisors} \} \overset{=}{\to} U_1  \overset{=}{\to}  U_2  \overset{=}{\to}  \cdots  \overset{=}{\to} U_M,
\end{equation}
so that we can regard $c$ as the uniquely determined $k$-rational point of $U_j$ under \eqref{eqn:cvn}, for each of $0 \leq j \leq M$.

\medskip

For each vertex $v \in S_M$, let $y^{v} = (y_1^v, \ldots, y_n ^v)$ be a coordinate system given on the open neighborhood $U_M ^v:= S_M \setminus \partial ^v S_M$, where $\partial^v S_M$ is the union (or the set) of divisors $D$ of $\partial S_M$ such that $v \not \in D$. 

We define the modified coordinate system $y^{v, c} := (y_1 ^{ v, c}, \ldots, y_n ^{v, c})$ by
\begin{equation}\label{eqn:3.2.2}
y_j ^{v, c}:= y_j ^v / y_j ^v (c).
\end{equation}
This defines the morphism (M. Levine \cite[p.322]{Levine moving})
\begin{equation}\label{eqn:3.2.2-1}
\lambda ^{v, c}: (\square_{\psi} ^n)_0 \to S_M
\end{equation}
given by inverting $y^{v,c}$ over $(\square_{\psi}^n)_0$ followed by the open immersion $U_M ^v \hookrightarrow S_M$.

They induce the ``internal" divisors defined by $\{ y_j ^{ v, c} = 1 \}$.

\medskip

We continue to follow M. Levine \cite[\S 4.7]{Levine moving}. For each vertex $v \in S_M$, let $\partial_v S_M$ be the divisors of $\partial S_M$ that \emph{do contain} $v$, with the induced order from $\partial S_M$. The ordered inclusion $\partial_v S_M \subset \partial S_M$ induces the morphism of complexes denoted by
$$
\eta^v: (S_M, \partial_v S_M) \to (S_M, \partial S_M).
$$

Let $\partial^+ (\square_{\psi} ^n)_0$ be the divisors $\{ D_i:= \{ y_i = 0 \} \ | \ 1 \leq i \leq n \}$. For any index set $I = \{ i_1 < \cdots < i_r\}\subset \{ 1, \ldots, n \}$, the coordinates $(y_1, \ldots, \check{y}_{i_1}, \ldots, \check{y}_{i_r}, \ldots, y_n)$ give the isomorphism $\iota_I: \left(\square_{\psi} ^{ n - | I |} \right)_0 \to D_I$. We have the natural morphism of complexes of degree $-n$
$$
\rho_n^+ : ( (\square_{\psi} ^\bullet)_0, d_+) \to ( (\square_{\psi} ^n)_0, \partial ^+ (\square_{\psi} ^n)_0)
$$
defined by the formal sum of the maps over all $I = \{ i_1 < i_2< \cdots < i_r \} \subset \{ 1, \ldots, n \}$
$$
(-1)^{ \sum_j i_j + nr}  \iota_{I} : \left( \square_{\psi} ^{ n-r}\right)_0 \to D_I.
$$

Define the map $\phi_+ ^{v,c}$ to be the composite
$$
( (\square_{\psi}^{\bullet})_0, d_+) \overset{\rho_n ^+}{\to} ( (\square_{\psi} ^n)_0, \partial ^+  (\square_{\psi} ^n)_0) \overset{ \lambda ^{v,c}}{\to} (S_M, \partial_v S_M) \overset{\eta^v}{\to} (S_M, \partial S_M).
$$
Similarly, we have $\phi^{v,c} : (  (\square_{\psi} ^\bullet)_0, d) \to (S_M, \partial S_M).$

\begin{defn}[{\cite[(4.4)]{Levine moving}}]\label{defn:4.4} Define $\phi^c: (  (\square_{\psi} ^\bullet)_0, d) \to (S_M, \partial S_M)$ to be the formal finite sum
$$
\phi^c:= \sum_v \varepsilon (v) \phi^{v,c}
$$
taken over all vertices $v \in S_M$. 
\end{defn}

\subsubsection{Higher level cubical subdivision}

Let $N$ be the number of distinguished divisors on $S_M$ for a tower $\pi: S_M \to \cdots \to S_0$ of blow-ups of faces as before in Definition \ref{defn:face blow}. Recall that (M. Levine \cite[p.327]{Levine moving}) a choice of a set map $\tau: \{ 1, \ldots, N \} \to \{ 1, \ldots, 2n \}$ such that $\pi ( (\partial S_M)_j) \subset (\partial S_0)_{\tau (j)}$ induces the associated morphism of complexes denoted by $\pi_* : (S_M, \partial S_M ) \to (S_0, \partial S_0),$ thus we have the composite
\begin{equation}\label{eqn:Le th}
\pi_* \circ \phi^c: ((\square_{\psi} ^{\bullet})_0, \partial (\square_{\psi} ^{\bullet})_0) \overset{\phi^c}{\to} (S_M, \partial S_M) \overset{\pi_*}{\to } (S_0, \partial S_0).
\end{equation}

This notation $\pi_*$ awkwardly looks like a push-forward, but it is not, and its restriction to $S_M$ (in degree $0$) is just the morphism $\pi$, for instance.
\medskip

The morphism \eqref{eqn:Le th} extends uniquely to the morphism of complexes (M. Levine \cite[(4.8)]{Levine moving})
\begin{equation}\label{eqn:Levine 4.8}
\Phi_\pi ^c: (\square_{\psi} ^{\bullet}, d) \to (S_0, \partial S_0),
\end{equation}
and this \eqref{eqn:Levine 4.8} is independent of the choice of the set map $\tau$ up to a chain homotopy (M. Levine \cite[Lemma 2.9]{Levine moving}). In the special case when we take $M=0$, we have $\pi= {\rm Id}$, which gives the special case
\begin{equation}\label{eqn:Levine 4.8-}
\Phi_{\rm Id} ^c: (\square_{\psi} ^{\bullet}, d)  \to (S_0, \partial S_0).
\end{equation}

\begin{defn}
Let $\pi: S_M \to \cdots \to S_0$ be a tower of blow-ups of faces as before. Choose a general $k$-rational point $c \in S_0 \setminus \{ \mbox{\rm faces} \}$. By construction, the morphism of complexes $\Phi_{\pi} ^c$ in \eqref{eqn:Levine 4.8} restricted to $\square_{\psi}^n$, can be identified with a finite $\mathbb{Z}$-linear combination of maps of the form
$$
\square_{\psi}^n  \dashrightarrow S_M \overset{\pi}{\to} S_0
$$
over the vertices $v \in S_M$, where the first broken arrow is a rational map, and the composition is a solid morphism.

\medskip

Let $Y$ be a $k$-scheme of finite type. For a cycle ${Z}$ on ${Y} \times S_0$, define \emph{the level $M$ cubical subdivision} of ${Z}$ on ${Y} \times S_M$ to be the strict transform
\begin{equation}\label{eqn:higher subdiv}
{\rm sd}_c ^M ({Z} ) := (\Phi_{\pi} ^c)^! ({Z}) =  \sum_{v\in S_M} \varepsilon (v) (\pi_*  \circ \phi^{v,c})^! ({Z}) \ \ \mbox{ on } {Y} \times \square_{\psi}^n,
\end{equation}
where the sum is taken over all vertices $v$ of $S_{M}$. When $M=0$, we have the equality $${\rm sd}_c ^0 ({Z}) = {\rm sd}_c ({Z})$$ with the cubical subdivision ${\rm sd}_c ({Z})$ of Definition \ref{defn:cub sd}.
 \qed
\end{defn}

\begin{remk}
The Figure 3 (cf. S. Bloch \cite[Figure 2, p.540]{Bloch moving}) illustrates the intuition of what is being done in the case of a single blow-up $S_1 \to S_0 = \mathbb{A}^2$ of a vertex. Here, we blow up the upper left corner vertex of the rectangle $S_0$ on the right to obtain the exceptional divisor, drawn by the thick upper-left sloped line of the left pentagon. The choice of $c$ in the interior of the rectangle gives the cubical faces near the vertices of the pentagon.
\begin{center}{Figure 3}
\begin{tikzpicture}[line cap=round,line join=round,>=triangle 45,x=1.0cm,y=1.0cm]
\clip(-0.56,-0.64) rectangle (6.5,2.9);
\fill[line width=2.pt,fill opacity=0.10000000149011612] (0.,0.) -- (0.,1.38) -- (1.04,2.44) -- (2.56,2.44) -- (2.58,0.) -- cycle;
\fill[line width=2.pt,fill opacity=0.10000000149011612] (4.,2.36) -- (4.,0.) -- (6.28,0.) -- (6.28,2.34) -- cycle;
\draw [line width=2.pt] (0.,0.)-- (0.,1.38);
\draw [line width=3.6pt] (0.,1.38)-- (1.04,2.44);
\draw [line width=2.pt] (1.04,2.44)-- (2.56,2.44);
\draw [line width=2.pt] (2.56,2.44)-- (2.58,0.);
\draw [line width=2.pt] (2.58,0.)-- (0.,0.);
\draw [line width=2.pt] (4.,2.36)-- (4.,0.);
\draw [line width=2.pt] (4.,0.)-- (6.28,0.);
\draw [line width=2.pt] (6.28,0.)-- (6.28,2.34);
\draw [line width=2.pt] (6.28,2.34)-- (4.,2.36);
\draw [->,line width=2.pt] (2.86,1.18) -- (3.66,1.18);
\draw [line width=2.pt,dash pattern=on 5pt off 5pt] (5.420241594214049,2.347541740401631)-- (5.44,0.);
\draw [line width=2.pt,dash pattern=on 5pt off 5pt] (4.,0.8)-- (6.34,0.76);
\draw [line width=2.pt,dash pattern=on 5pt off 5pt] (0.,0.7)-- (1.52,0.94);
\draw [line width=2.pt,dash pattern=on 5pt off 5pt] (1.52,0.94)-- (1.62,2.44);
\draw [line width=2.pt,dash pattern=on 5pt off 5pt] (1.52,0.94)-- (0.49519317975693816,1.8847161255214946);
\draw [line width=2.pt,dash pattern=on 5pt off 5pt] (1.52,0.94)-- (2.5722915686933154,0.9404286194155189);
\draw [line width=2.pt,dash pattern=on 5pt off 5pt] (1.52,0.94)-- (1.54,0.);
\end{tikzpicture}
\end{center}
\qed
\end{remk}

Between $\Phi_{\pi}^c$ and $\Phi_{\rm Id} ^c$ of \eqref{eqn:Levine 4.8} and \eqref{eqn:Levine 4.8-}, we have:

\begin{prop}[{\cite[Proposition 4.12]{Levine moving}} cf. {\cite[Proposition (3.4.2)]{Bloch moving}}]\label{prop:Levine homo}
Let $S_0:= \mathbb{A}^n$ with the $2n$ distinguished divisors. 
Let $\pi: S_M \to \cdots  \to S_0$ be a tower of blow-ups of faces. Let $c'= (c, c_{n+1}) \in (\mathbb{A}^1 \setminus \{ 0, 1 \})^{n+1}$. 

Then there is a homotopy $H_0 (c')$ between $\Phi_{\pi} ^c$ and $\Phi_{\rm Id} ^c$.
\end{prop}

We sketch part of the construction of the homotopy $H_0 (c')$ as we want to convince ourselves that the above homotopy induces a ``homotopy cycle" whose boundary is the difference of two cubical subdivisions at different levels of the blow-up tower.

The sketch follows the proof of M. Levine \cite[Proposition 4.12]{Levine moving}. Let $T_0:= \square_{\psi}^{n+1}$. We have the closed immersion $\iota: S_0 \hookrightarrow T_0$ given by $\{ y_{n+1} = 1 \}$, while we also have the flat projection $pr: T_0 \to S_0$, that ignores the last coordinate $y_{n+1}$. Apparently $pr \circ \iota = {\rm Id}_{S_0}$.

For the given blow-up tower $\pi: S_M \to \cdots \to S_0$, if $S_1 \to S_0$ is given by blowing up a face $F \subset S_0$, then via $\iota: S_0 \hookrightarrow T_0$, we regard $F$ as a face of $T_0$, and take the blow-up $T_1 \to T_0$. Here $T_1$ contains $S_1$ as the strict transform of $\iota (S_0)$. Inductively, we obtain $T_{i+1} \to T_i$ by blowing-up the face of $S_i$ obtained to get $S_{i+1} \to S_i$. This gives the tower $\pi': T_M \to \cdots \to T_0$ into which the tower $\pi$ is embedded, where the distinguished divisors in $ \partial T_M$ except for the strict transforms of $\{ y_{n+1} = \epsilon \}$, with $\epsilon \in \{ 0, 1 \}$, are in one-to-one correspondence with the divisors in $\partial S_M$. Let $\partial_{-} T_M$ be those divisors unrelated to $\{ y_{n+1} = \epsilon \}$ for $\epsilon \in \{0, 1 \}$. 

Then, the choice of $c' = (c, c_{n+1})$ gives a composite
\begin{equation}\label{eqn:3.2.11-1}
pr_* \circ \phi_h ^{c'} : ((\square_{\psi} ^{\bullet})_0, \partial ( \square_{\psi}^{\bullet})_0 ) \to (T_M, \partial_{-} T_M ) \to (S_0, \partial S_0),
\end{equation}
where $\phi_h ^{c'}$ is defined by Definition \ref{defn:4.4}, with $(\partial_M, \partial S_M)$ there replaced by $(T_M, \partial_- T_M)$, and \eqref{eqn:3.2.11-1} uniquely extends to (M. Levine \cite[(4.8)]{Levine moving})
$$
H_0 (c') : (\square_{\psi} ^{\bullet}, \partial  \square_{\psi}^{\bullet} ) \to (T_M, \partial_{-} T_M ) \to (S_0, \partial S_0).
$$
Identifying $S_j$ with the image of the closed immersion $S_j \hookrightarrow T_j$, we have (\cite[p.329]{Levine moving}) $d H_0 (c') = \Phi_{\pi} ^c - \Phi_{{\rm Id}} ^c$. 

\begin{remk}\label{remk:5.2.12'}
A point is that, for a cycle ${Z}$ on ${Y} \times S_0$, we have the flat pull-back $pr^* ({Z})$ on ${Y} \times T_0$, and thus the strict transform $H_0 (c')^! (pr^* {Z})$. When $c'$ is general, its boundary is $(\Phi_{\pi} ^c )^! ({Z}) - (\Phi_{{\rm Id}} ^c)^! ({Z})$ by Proposition \ref{prop:Levine homo}.
\qed
\end{remk}

\subsubsection{General position}

Let $Y$ be a $k$-scheme of finite type.
Recall the following Theorem \ref{thm:bloch-levine} from \cite[Theorem (0.3), p.70]{Levine JPAA} (cf. S. Bloch \cite[Lemma 2.1.1, Theorem 2.1.2]{Bloch moving} or M. Levine \cite[Theorem 6.12, p.343]{Levine moving}), which is based on M. Spivakovsky \cite{Spivakovsky} on Hironaka's polyhedral game:

\begin{thm}\label{thm:bloch-levine}
Let $k$ be an infinite field. 
Let $p: Z \to S_0$ be a morphism of $k$-schemes of finite type. Suppose 
$$
p (Z_i) \not \subset \{ \mbox{the $2n$ distinguished divisors} \}
$$
for every irreducible component $Z_i$ of $Z$. 

Then there exists a tower $\pi: S_M \to \cdots \to S_0$ of blow-ups of faces such that, for a general $k$-rational point $c \in S_0$ not contained in any face, each component of the strict transform $(\Phi_{\pi} ^c) ^! Z$ on $\square_{\psi}^n$ intersects properly with all the faces of $\square_\psi ^n$. 

Here, we regard $\Phi_{\pi} ^c$ as a finite $\mathbb{Z}$-linear combination of maps $\coprod_v \square_{\psi} ^n \dashrightarrow S_M \to S_0$, where $v$ runs over all vertices of $S_M$.
\end{thm}

\subsection{The proof}\label{sec:proof localization}

We now prove Theorem \ref{thm:localization}. Let $k$ be an arbitrary field.

Consider $\mathcal{C}_{\bullet}={\rm coker} (\rho^Y_U : z_d (Y, \bullet) \to z_d (U, \bullet)) = \frac{ z_d (U , \bullet)}{ \rho^Y_U ( z_d ({Y}, \bullet))}$. We prove that $\mathcal{C}_{\bullet}$ is acyclic. The statement is trivial when $U= \emptyset$, so we may suppose $U \not = \emptyset$.

Let $n=0$. For each cycle ${Z} \in z_d (U, 0)$, its closure over ${Y}$ belongs to $z_d ({Y}, 0)$. Thus the restriction map
$z_d ({Y}, 0) \to z_d (U, 0)$ is surjective, thus ${\rm H}_0 (\mathcal{C}_{\bullet})= 0$. 

\medskip

We now suppose $n>0$. We prove that ${\rm H}_n (\mathcal{C}_{\bullet}) = 0$ in two steps.

\medskip

\noindent \textbf{Step 1:} \emph{Reduction to the case of infinite base fields.}

\medskip

Temporarily suppose that Theorem \ref{thm:localization} holds for all \emph{infinite} base fields. 

Let $k$ be a finite field. We use the infinite pro-$\ell$ extension trick as in S. Bloch \cite[Lemma 4.1.2]{Bloch moving}. The argument goes as follows.

For any prime $\ell$, we have an extension $k \hookrightarrow k_{\ell}$, where $k_{\ell}$ is an infinite field obtained as the colimit of a sequence of finite extensions $k \subset  \cdots \subset k_i \subset k_{i+1} \subset \cdots$ of fields such that each degree $[k_i:k]$ is a power of $\ell$. Let ${\mathcal{C}_{\bullet}}_{k_{\ell}}$ be the complex obtained via the base change ${Y}_{k_{\ell}} \to {Y}$.

Then ${\rm H}_n ({\mathcal{C}_{\bullet}}_{k_{\ell}}) = 0$ by the given assumption. By Corollary \ref{cor:pbpf degree}, every class $x \in {\rm H}_n (\mathcal{C}_{\bullet})$ is a torsion such that $\ell ^{N_1} x = 0$ for some $N_1 >0$. Since $\ell$ was an arbitrary prime (rename the first one to $\ell_1$), if we repeat the above with another prime $\ell_2 \not = \ell_1$, then we have $\ell_2^{N_2} x = 0$ for some $N_2 >0$. Since $\ell_1 ^{N_1}$ and $\ell_2 ^{N_2}$ are coprime, there exist $ a, b \in \mathbb{Z}$ such that $a \ell_1 ^{N_1} + b \ell_2 ^{N_2} = 1$. Hence $x = 1 \cdot x = (a \ell_1 ^{N_1} + b \ell_2 ^{N_2} ) x = a \ell_1 ^{N_1}x + b \ell_2 ^{N_2} x=0 + 0 = 0$. This implies Theorem \ref{thm:localization} for a finite field $k$, assuming that the theorem holds for all infinite fields.

\medskip

 It remains to prove the theorem when $k$ is infinite.

\noindent \textbf{Step 2:} Suppose that $k$ is an infinite field. 

\medskip

We continue our convention of using $\square^n_{\psi}$, where $\square _{\psi} = \mathbb{A}^1$ with the faces $\{0, 1 \}$, via the isomorphism $\square \simeq \square_{\psi}$ given by the automorphism $y \mapsto \frac{y}{y-1}$ of $\mathbb{P}_k ^1$. Write $S_0:= \square^n_{\psi}$ as before.

For $n>0$, let $\xi \in {\rm H}_n (\mathcal{C}_{\bullet})$ be an arbitrary class. By the normalization theorem (Theorem \ref{thm:normalization}), we can represent this class by a cycle ${Z} \in z_d (U , n)$ such that $\partial_i ^{\epsilon} ({Z}) $ is trivial in $\mathcal{C}_{n-1}$ for all $1 \leq i \leq n$ and $\epsilon \in \{ 0, 1\}$. 

Apply the cubical subdivision with respect to a general $k$-rational point $c \in S_0 \setminus \{ \mbox{\rm faces} \}$ (Proposition \ref{prop:cub sd}) so that ${Z}$ is equivalent to ${W}:= {\rm sd}_c ({Z}) = ^{\dagger}{\rm sd}_c ^0 ({Z}) \in z_d (U , n)$ in $\mathcal{C}_n$ modulo $\partial \mathcal{C}_{n+1}$, where $\dagger$ holds by Lemma \ref{lem:sd^0 comparison}.
This ${W}$ represents $\xi$ as well.

Consider the closure $\ov{{W}}$ of the cycle ${W}$ in ${Y} \times S_0$. Here $\overline{{W}}$ \emph{may not} intersect all faces of ${Y} \times S_0$ properly. We improve it with the ``blowing-up faces of a cube" technique discussed previously.

Indeed, by Theorem \ref{thm:bloch-levine} applied to the closure $\overline{{W}}$ on ${Y} \times S_0$, there exists a tower $\pi: S_M \to \cdots \to S_0$ of blow-ups of faces, and a general $k$-rational point $c \in S_0 \setminus \{ \mbox{\rm faces} \}$, such that the strict transform ${A}:= (\Phi_{\pi} ^c)^! \ov{ {Z}}$ on ${Y} \times \square_{\psi}^n$ on the level $S_M$ is in $z_d ({Y}, n)$. This ${A}$ is also the Zariski closure over ${Y}$ of the strict transform ${\rm sd}_c ^M ({Z}) =(\Phi_{\pi}^c)^! {Z}$. 

By Proposition \ref{prop:Levine homo} and Remark \ref{remk:5.2.12'}, modulo a boundary, the cycle ${\rm sd}_c ^0 ({Z}) = {W}$ is equivalent to ${\rm sd}_c ^M ({Z})$. This shows that ${\rm sd}_c ^M ({Z})$ also represents the original cycle class $\xi$. Since its closure ${A}$ belongs to $z_d ({Y}, n)$, it means $\xi = 0$ in ${\rm H}_n (\mathcal{C}_{\bullet})$. Since $\xi$ was arbitrary, this shows that $\mathcal{C}_{\bullet}$ is acyclic. This proves Theorem \ref{thm:localization}. \qed

\section{Zariski sheafifications and a flasque resolution}\label{sec:sheaf}

Let $Y$ be a $k$-scheme of finite type. Let ${\rm Op} (Y)$ be the category of open subschemes of $Y$.

In \S \ref{sec:sheaf}, we study the sheafification $\BGHz_d (\bullet)_Y$ of the cubical higher Chow presheaf 
$$
 \mathcal{P}_{\bullet, d}: U\in {\rm Op} (Y)  \mapsto z_d (U, \bullet),
 $$
 and describe the groups $\CH_d ({Y}, n)$ of Definitions \ref{defn:HCG} as hypercohomology groups, more specifically, 
\begin{equation}\label{eqn:hjysn}
 \CH_d ({Y} , n) \simeq \mathbb{H}_{\rm Zar} ^{-n} (Y, \BGHz_d (\bullet)_Y).
\end{equation}

\subsection{The topological support and a flasque sheaf}\label{sec:desc sheaf}

\begin{defn}\label{defn:top support}
Let $Y $ be a $k$-scheme of finite type and let $W \subset | Y|$ be a closed subset. We say an integral cycle ${Z}$ on ${Y}$ is \emph{supported in $W$}, if the topological support of $\mathcal{O}_{Z}$ is contained in $W$. A cycle $\alpha$ on $Y$ is \emph{supported in $W$}, if the union of the topological supports of the integral components of $\alpha$ is contained in $W$.\qed
\end{defn}

\begin{defn}
Let $Y$ be a $k$-scheme of finite type. 
Let $W \subset | Y| $ be a closed subset, and let $U:=Y \setminus W$.  
Define 
$$
 {G} _d  (W, n):= \ker ( \rho^Y_U:  z_d ({Y}, n) \to z_d (U, n)),
 $$
for the flat restriction map $\rho^Y_U$. This is the group of cycles in $z_d ({Y}, n)$ supported in $W \times \square^n$. 

Give the induced reduced closed subscheme structure on $W$, and let $\iota: W \hookrightarrow Y$ be the closed immersion. Via $\iota_*: z_d (W, n) \to z_d (Y, n)$, we may regard $z_d (W, n)$ as a subgroup. We note immediately that 
$$
G_d (W, n) = z_d (W, n)
$$
so that the sequence 
$$
0 \to z_d (W, n) \to z_d (Y, n) \to z_d (U, n)
$$
is exact.
\qed
\end{defn}

\medskip

We make the following apparent observation:

\begin{lem}\label{lem:topsupp_rel0}
Let $Y$ be a $k$-scheme of finite type. Let $W_1, W_2 \subset Y$ be closed subsets. 
\begin{enumerate}
\item If $W_1 \subset W_2$, then ${G}_d (W_1, n) \subset {G}_d (W_2, n)$.
\item We have ${G}_d  (W_1, n) \cap {G}_d  (W_2, n) = {G}_d (W_1 \cap W_2, n)$.
\end{enumerate}
\end{lem}

Lemma \ref{lem:flasque psh} and Proposition \ref{prop:flasque sh} below are analogues of what was already proven in the simplicial version of higher Chow cycles in S. Bloch \cite[Theorem (3.4), p.278]{Bloch HC}. We present here detailed arguments because ours for the cubical version don't follow as a special case of \emph{loc.cit.}, but we remark that the essential ideas are from there.

\begin{lem}\label{lem:flasque psh}
Let $Y$ be a $k$-scheme of finite type. Let ${\rm Op} (Y)$ be the category of open subschemes of $Y$. Let $n \geq 0$ be integers. Consider the association
\begin{equation}\label{eqn:sheaf small}
{\mathcal{S}}_n = {\mathcal{S}}_{n,d} :  U \in {\rm Op} (Y) \mapsto  \frac{ z_d  ( {Y} , n)}{{G}_d (Y\setminus U, n )},
\end{equation}
Then $\mathcal{S}_n$ is a flasque presheaf on $Y_{\rm Zar}$.
\end{lem}

\begin{proof}
For two Zariski open subsets $V \subset U$ of $Y$, we have closed subsets $Y \setminus U \subset Y \setminus V \subset Y$. They induce the natural inclusions (Lemma \ref{lem:topsupp_rel0})
$$
G_d (Y \setminus U,n) \subset G_d (Y \setminus V,n) \subset z_d  ({Y}, n).
$$
These in turn induce the natural surjection
$$
\rho_{V} ^U: \mathcal{S}_n (U) =  \frac{ z_d ( {Y} , n)}{G_d ( Y \setminus U, n) } \twoheadrightarrow \mathcal{S}_n (V) =  \frac{ z_d  ( {Y} , n)}{G_d ( Y \setminus V, n)}.
$$
When $U_3 \subset U_2 \subset U_1$ are open subsets of $Y$, that $\rho_{U_3} ^{U_1} = \rho_{U_3} ^{U_2} \circ \rho_{U_2} ^{U_1}$ is apparent. Hence $\mathcal{S}_n$ is a flasque presheaf.
\end{proof}

\begin{prop}\label{prop:flasque sh}
For $n \geq 0$, the presheaf $\mathcal{S}_n= \mathcal{S}_{n,d}$ of Lemma \ref{lem:flasque psh} is a flasque sheaf on $Y_{\rm Zar}$. 
\end{prop}

\begin{proof}
By Lemma \ref{lem:flasque psh}, it remains to show that $\mathcal{S}_n$ is a sheaf.

\medskip

By induction on the number of open sets in a finite open cover of $Y$, it is enough to check that for two open subsets $U, V \subset Y$, the sequence
\begin{equation}\label{eqn:sheafproof}
0 \to \mathcal{S}_n (U \cup V) \overset{\delta_0}{\to} \mathcal{S}_n (U) \oplus \mathcal{S}_n (V) \overset{\delta_1}{\to} \mathcal{S}_n (U \cap V)
\end{equation}
is exact, where $\delta_0 = (\rho_U ^{U \cup V}, \rho_V ^{U \cup V})$ and $\delta_1 (\bar{x}_1, \bar{x}_2) = \rho^U_{U\cap V} (\bar{x}_1) - \rho ^V _{U \cap V} (\bar{x}_2)$. 

\medskip

We first show $\delta_0$ is injective. 

Suppose that $\delta_0 (\bar{x}) = 0$ for some $\bar{x} \in \mathcal{S}_n (U \cup V)$. It means there is a cycle $x \in z_d ({Y}, n)$ representing $\bar{x}$ satisfying
\begin{equation}\label{eqn:exactness1}
x \in G_d (Y \setminus U,n ) \cap G_d (Y \setminus V, n).
\end{equation}
The intersection of the groups in \eqref{eqn:exactness1} is $G_d ((Y \setminus U) \cap (Y \setminus  V), n)= G_d (Y \setminus (U \cup V), n)$ by Lemma \ref{lem:topsupp_rel0}. Thus \eqref{eqn:exactness1} means $\bar{x} = 0$ in $\mathcal{S}_n (U \cup V)$. Thus $\delta_0$ is injective.

\medskip

That $\delta_1 \circ \delta_0 = 0$ is apparent, so that ${\rm im} (\delta_0) \subset \ker \delta_1$. 

\medskip

We prove ${\rm im} (\delta_0) \supset \ker \delta_1$. 

Let $(\bar{x}_1, \bar{x}_2) \in \ker \delta_1$. It means there exist cycles $x_1 , x_2 \in z_d  ( Y, n)$ representing $\bar{x}_1, \bar{x}_2$, respectively, satisfying
\begin{equation}\label{eqn:exactness2}
x_1 - x_2 \in G_d ( Y \setminus (U \cap V), n) = G_d ( (Y \setminus U) \cup (Y \setminus V), n).
\end{equation}

 Since $x_1 - x_2$ is supported in $ (Y \setminus (U \cap V)) \times \square^n$ by \eqref{eqn:exactness2}, the restriction $x_1|_{U\cap V} - x_2 |_{U \cap V}$ has the empty topological support. Hence $x_1|_U$ and $x_2|_V$ glue to give a cycle $x$ on $(U \cup V) \times \square^n$ such that $x|_U = x_1|_U$ and $x|_V = x_2 |_V$.
 
 Let $\tilde{x}$ be the Zariski closure of $x$ in $Y \times \square^n$. By construction, $\tilde{x}$ has no component supported in $(Y \setminus ( U \cup V)) \times \square^n$. We don't know yet whether $\tilde{x} \in z_d (Y, n)$, but we \textbf{claim} that it is so. We prove it in what follows.
 
 \medskip
 
 Define $\tilde{x}_U$ to be the closure of $x_1|_U$ in $Y \times \square^n$, and $\tilde{x}_V$ be the closure of $x_2|_V$ in $Y \times \square^n$. Here, by construction no component of $\tilde{x}_U$ is supported in $(Y \setminus U) \times \square^n$. Since $\tilde{x}_U$ is defined as the closure of the open restriction $x_1|_U$ of the (already) admissible cycle $x_1 \in z_d (Y, n)$, we have $\tilde{x}_U \in z_d (Y, n)$. Similarly, no component of $\tilde{x}_V$ is supported in $(Y \setminus V) \times \square^n$ and $\tilde{x}_V \in z_d (Y, n)$. 
 
We can write
\begin{equation}\label{eqn:S sheaf 1}
\tilde{x} = \tilde{x}_U + \Delta_U = \tilde{x}_V + \Delta_V
\end{equation}
for some cycles $\Delta_U$ and $\Delta_V$ on $Y \times \square^n$ such that $\Delta_U$ is supported in $(Y \setminus U) \times \square^n$, $\Delta_V$ is supported in $( Y \setminus V) \times \square^n$, while we do not know whether $\Delta_U, \Delta_V$ are admissible. 
 
 At least by \eqref{eqn:S sheaf 1}, we know 
 \begin{equation}\label{eqn:S sheaf 2}
 \Delta_U - \Delta_V =  \tilde{x}_V -  \tilde{x}_U \in z_d (Y, n).
 \end{equation}
  If there is any integral component $W$ of $\Delta_U$ that is \emph{not} admissible, then to have the admissibility of $\Delta_U - \Delta_V$ in \eqref{eqn:S sheaf 2}, $\Delta_V$ must have $W$ as a component to cancel out the first $W$ from $\Delta_U$. Hence such $W$ has the support in $ (( Y \setminus U) \times \square^n) \cap ((Y \setminus V) \times \square^n) = (Y \setminus (U \cup V)) \times \square^n$. The same holds for non-admissible integral components of $\Delta_V$ as well.

 Thus we can write
 \begin{equation}\label{eqn:S sheaf 3}
 \Delta_U = \Delta_U' + \Delta_{U,V}
 \end{equation}
 for some cycles $\Delta_U'$ and $\Delta_{U,V}$ on $Y \times \square^n$, such that $\Delta_U ' \in z_d (Y, n)$, and the support of $\Delta_{U,V}$ is in $(Y \setminus  (U \cup V)) \times \square^n$. If any component of $\Delta_U'$ is supported in $(Y \setminus (U \cup V)) \times \square^n$, then move it into $\Delta_{U,V}$, so that we may assume $\Delta_U'$ is admissible and the supports of its components are not in $(Y \setminus (U \cup V)) \times \square^n$.

Combining \eqref{eqn:S sheaf 1} and \eqref{eqn:S sheaf 3}, we have
$$
\tilde{x} = \tilde{x}_U + \Delta_U ' + \Delta_{U,V}.
$$
But by construction, no component of $\tilde{x}$ is supported in $(Y \setminus (U \cup V)) \times \square^n$. Hence, we must have $\Delta_{U,V} = 0$. In particular, this means $\tilde{x} = \tilde{x}_U + \Delta_U' \in z_d (Y, n)$, i.e. $\tilde{x}$ is admissible, proving the \textbf{claim}.

\medskip

Let $\bar{x} \in \mathcal{S}_n (U \cup V)$ be the class of $\tilde{x}$. Because $\tilde{x} = \tilde{x}_U + \Delta_U' $ while $\Delta_U' \in G_d (Y \setminus U, n)$, we have $\tilde{x}_U \equiv \rho^{U \cup V}_U (\bar{x})$ in $\mathcal{S}_n (U)$. By symmetry, $\tilde{x}_V \equiv \rho^{U \cup V}_V (\bar{x})$ in $\mathcal{S}_n (V)$. 

On the other hand, $\bar{x}_1 \equiv \tilde{x}_U$ in $\mathcal{S}_n (U)$ and $\bar{x}_2 \equiv \tilde{x}_V$ in $\mathcal{S}_V$ by the constructions of $\tilde{x}_U$ and $\tilde{x}_V$. Thus we have $\delta_0 (\bar{x}) = (\bar{x}_1, \bar{x}_2)$. This proves the exactness of \eqref{eqn:sheafproof} at the middle term of the sequence.

\medskip

Hence \eqref{eqn:sheafproof} is exact. Thus $\mathcal{S}_n= \mathcal{S}_{n,d}$ is a sheaf.
\end{proof}

\begin{cor}\label{cor:flasque sh}
For each pair of open subsets $U, V \subset Y$, we have the short exact sequence
$$
0 \to \mathcal{S}_{n,d} (U \cup V) \overset{\delta_0}{\to} \mathcal{S}_{n,d} (U) \oplus \mathcal{S}_{n,d} (V) \overset{\delta_1}{\to} \mathcal{S}_{n,d} (U \cap V) \to 0.
$$
\end{cor}

\begin{proof}
By the exact sequence \eqref{eqn:sheafproof} of Proposition \ref{prop:flasque sh}, it remains to check the surjectivity of $\delta_1$. Since ${S}_{n,d}$ is flasque by Lemma \ref{lem:flasque psh}, the restriction map $\mathcal{S}_{n,d} (U) \to \mathcal{S}_{n,d} (U \cap V)$ is already surjective. Thus so is the map $\delta_1$.
\end{proof}

Theorem \ref{thm:localization} and Proposition \ref{prop:flasque sh} together now imply:

\begin{prop}\label{prop:flasque rep}
Let $Y$ be a $k$-scheme of finite type. Then we have a natural morphism of complexes of abelian sheaves on $Y_{\rm Zar}$
\begin{equation}\label{eqn:SP_sh}
\mathcal{S}_{\bullet, d}  \to \BGHz_d ( \bullet)_Y,
\end{equation}
and this is a quasi-isomorphism.
\end{prop}

\begin{proof}
For each nonempty open subset $U \subset Y$, the sequence
$$
0 \to {G}_d (Y \setminus U, n) \to z_d ({Y} , n) \to z_d  (U, n),
$$ 
is exact. Hence we deduce the induced injective homomorphism
\begin{equation}\label{eqn:global qi-local}
\mathcal{S}_{n,d} (U) = \frac{ z_d ({Y} , n)}{G_d (Y \setminus U, n)} \to \mathcal{P}_{n,d} (U) =z_d ( U, n).
\end{equation}
Since they are equivariant under the boundary maps $\partial$, collecting them over all $n \geq 0$, we have the injective homomorphism of complexes of presheaves on $Y$
\begin{equation}\label{eqn:global qi-local2}
 \mathcal{S}_{\bullet, d}  \to \mathcal{P}_{\bullet,d},
\end{equation}
and they in turn induce the injective morphisms \eqref{eqn:SP_sh} of complexes of sheaves on $Y_{\rm Zar}$
because $\BGHz_d ( \bullet)_Y$ is the sheafification of ${\mathcal{P}}_{\bullet,d}$, while $\mathcal{S}_{\bullet, d}$ is already a sheaf by Proposition \ref{prop:flasque sh}. 

By Theorem \ref{thm:localization}, the cokernels of \eqref{eqn:global qi-local2} are acyclic over each nonempty open subset $U \subset Y$. Hence the morphism \eqref{eqn:global qi-local2} is a quasi-isomorphism for any nonempty open $U \subset Y$. These quasi-isomorphisms induce a quasi-isomorphism of the stalks of \eqref{eqn:SP_sh} at each scheme point $y \in Y$. Thus, the morphism \eqref{eqn:SP_sh} is a quasi-isomorphism of complexes of abelian sheaves on $Y_{\rm Zar}$.
\end{proof}

The groups $\CH_d ({Y}, n)$ in Definitions \ref{defn:HCG} can be related to the hypercohomology groups of the complexes of sheaves:

\begin{thm}\label{thm:two BGH}
Let $Y$ be a $k$-scheme of finite type. Then we have an isomorphism
\begin{equation}\label{eqn:semi-loc case}
 \CH_d  ({Y}, n) =  \mathbb{H}_{\rm Zar} ^{-n} (Y, \BGHz_d (\bullet)_Y).
\end{equation}
\end{thm}

\begin{proof}
The morphism $\mathcal{S}_{\bullet, d}  \to \BGHz_d ( \bullet)_Y$ in \eqref{eqn:SP_sh} is a quasi-isomorphism (Proposition \ref{prop:flasque rep}), while $\mathcal{S}_{\bullet, d}$ is a complex of flasque sheaves (Proposition \ref{prop:flasque sh}). Hence we have
$$
\mathbb{H}_{\rm Zar} ^{-n} (Y, \BGHz_d ( \bullet)_Y)= \mathbb{H}_{\rm Zar} ^{-n} (Y, \mathcal{S}_{\bullet,d} ) = {\rm H}^{-n} \Gamma ( Y, \mathcal{S}_{\bullet,d })= {\rm H}_{n} \Gamma ( Y, \mathcal{S}_{\bullet, d}).
$$

On the other hand, for each $n \geq 0$ by definition
$$
\Gamma (Y,  \mathcal{S}_{n,d}) =  \frac{z_d  (Y, n)}{G_d (Y \setminus Y,n)} = z_d ( Y, n),
$$
so that the $n$-th homology groups of the complex $\Gamma (Y,  {\mathcal{S}}_{\bullet, d})$ is precisely $\CH_d ({Y}, n)$. This proves \eqref{eqn:semi-loc case}. 
\end{proof}

\subsection{Zariski descent}\label{sec:Zariski desc}

We deduce that our groups satisfy the Zariski descent, i.e. the Zariski Mayer-Vietoris property:

\begin{thm}\label{thm:Zar desc 1}
Let $Y$ be a $k$-scheme of finite type. Let $U$ and $V \subset Y$ be Zariski open subsets.

Then we have the short exact sequence of complexes of abelian groups
\begin{equation}\label{eqn:sheafproof1}
0 \to \mathcal{S}_{\bullet,d} (U \cup V) \overset{\delta_0}{\to} \mathcal{S}_{\bullet,d}(U) \oplus \mathcal{S}_{\bullet,d} (V) \overset{\delta_1}{\to} \mathcal{S}_{\bullet,d}(U \cap V) \to 0.
\end{equation}
The complexes $\BGHz_d( \bullet)_Y$ of sheaves satisfy the Zariski Mayer-Vietoris on $Y_{\rm Zar}$.
\end{thm}

\begin{proof}
The exact sequences \eqref{eqn:sheafproof1} follow from Corollary \ref{cor:flasque sh}. The long exact sequences associated to the short exact sequences in \eqref{eqn:sheafproof1} prove the Zariski Mayer-Vietoris property for $\mathcal{S}_{\bullet,d}$. Since we have a quasi-isomorphism $\mathcal{S}_{\bullet,d} \to \BGHz_d ( \bullet)_Y$ by Proposition \ref{prop:flasque rep}, we deduce the Zariski Mayer-Vietoris property for  $\BGHz_d (\bullet)_Y$. 
\end{proof}

\end{document}